\documentclass[12pt]{amsart}
\usepackage{}

\usepackage{amsmath}
\usepackage{amsfonts}
\usepackage{amssymb}

\usepackage[all,cmtip]{xy}           %xypic macro for latex2.09
\usepackage{bm}
\usepackage{bbm}

\usepackage{bbding}
\usepackage{txfonts}
\usepackage{amscd}
\usepackage[all]{xy}
\usepackage[shortlabels]{enumitem}
\usepackage{ifpdf}
\ifpdf
  \usepackage[colorlinks,final,backref=page,hyperindex]{hyperref}
\else
  \usepackage[colorlinks,final,backref=page,hyperindex,hypertex]{hyperref}
\fi
\usepackage{tikz}
\usepackage[active]{srcltx}
\usepackage{caption}

\makeatletter

\newtheorem{thm}{Theorem}[section]
\newtheorem{lem}[thm]{Lemma}
\newtheorem{cor}[thm]{Corollary}
\newtheorem{pro}[thm]{Proposition}
\newtheorem{ex}[thm]{Example}
\newtheorem{rmk}[thm]{Remark}
\newtheorem{defi}[thm]{Definition}

\newcommand {\emptycomment}[1]{}

\newcommand{\be }{\begin{equation}}
\newcommand{\ee }{\end{equation}}

\newcommand{\huaA}{\mathcal{A}}%{{\mathcal{F}}}%{\mathcal{A}}
\newcommand{\huaL}{\mathcal{L}}
\newcommand{\huaR}{\mathcal{R}}

\newcommand{\huaP}{\mathcal{P}}
\newcommand{\huaI}{\mathcal{I}}

\newcommand{\huaO}{{\mathcal{O}}}

\newcommand{\frkd}{\mathfrak d}

\newcommand{\frkr}{\mathfrak r}

\newcommand{\frkB}{\mathfrak B}

\newcommand{\frkL}{\mathfrak L}

\newcommand{\Id}{{\rm{Id}}}

\newcommand{\br}[1]{   [ \cdot,    \cdot  ]   }

\newcommand{\End}{\mathrm{End}}

\newcommand{\ad}{\mathrm{ad}}

\newcommand{\Img}{\mathrm{Im}}

\begin{document}

\title{Zinbiel bialgebras, relative Rota-Baxter operators and the related Yang-Baxter Equation}

\author{You Wang}
\address{Department of Mathematics, Jilin University, Changchun 130012, Jilin, China}
\email{wangyou20@mails.jlu.edu.cn}

%\date{\today}

\begin{abstract}
  In this paper, we first introduce the notion of a Zinbiel bialgebra and show that Zinbiel bialgebras, matched pairs of Zinbiel algebras and Manin triples of Zinbiel algebras are equivalent. Then we study the coboundary Zinbiel bialgebras, which leads to an analogue of the classical Yang-Baxter equation. Moreover, we introduce the notions of quasi-triangular and factorizable Zinbiel bialgebras as special cases. A quasi-triangular Zinbiel bialgebra can give rise to a relative Rota-Baxter operator of weight $-1$. A factorizable Zinbiel bialgebra can give a factorization of the underlying Zinbiel algebra.
  As an example, we define the Zinbiel double of a Zinbiel bialgebra, which enjoys a natural factorizable Zinbiel bialgebra structure. Finally, we introduce the notion of quadratic Rota-Baxter Zinbiel algebras, as the Rota-Baxter characterization of factorizable Zinbiel bialgebras. We show that there is a one-to-one correspondence between quadratic Rota-Baxter Zinbiel algebras and factorizable Zinbiel bialgebras.
\end{abstract}


\maketitle

\tableofcontents

\allowdisplaybreaks

%\end{document}
%\vspace{1mm}

\section{Introduction}
The purpose of this paper is to establish the bialgebra theory for Zinbiel algebras. Moreover, we also develop the quasi-triangular and factorizable theories for Zinbiel bialgebras.

\subsection{Zinbiel algebras}

Zinbiel algebras (or dual Leibniz algebras) was initiated by Loday (\cite{Loday}). The cup product on the cohomology group of a Leibniz algebra forms a Zinbiel algebra structure.
It is also known as pre-commutative algebras, Tortkara algebras or chronological algebras. Zinbiel algebras are koszul dual to Leibniz algebras \cite{Balavoine} and   a differential graded algebra was also  constructed in \cite{Balavoine} to obtain the (co)homology of Zinbiel algebras by operad theory.
Zinbiel algebras appeared in many fields of mathematics and mathematical physics, such as
operad theory (\cite{Balavoine,GK,LV}), rack cohomology (\cite{CFLM}), Leibniz algebras and Leibniz cohomology (\cite{Loday,Balavoine,MS,Saha}), deformation theory (\cite{Yau}) and multiple zeta values (\cite{Chapoton}). The classification of low dimensional Zinbiel algebras was investigated in \cite{AOK,ALO}. More recent studies on Zinbiel algebras can be seen in \cite{Tower2,CFLM,GGZ,Tower1}.

There is a close relationship among Lie algebras, associative algebras, commutative algebras, pre-Lie algebras, dendriform algebras and Zinbiel algebras as follows (in the sense of commutative diagram of categories). See \cite{Aguiar,Bai} for more details.
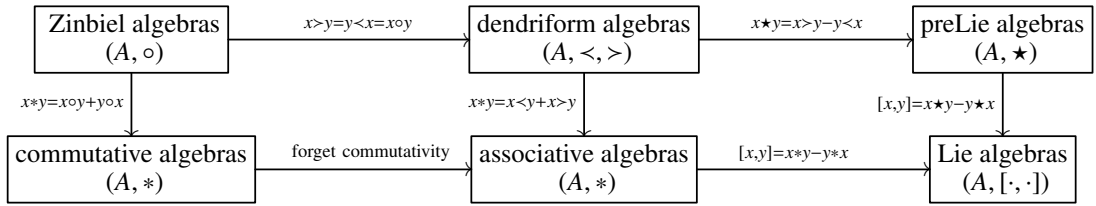
\begin{equation*}
{\footnotesize
\xymatrix@C=18ex{
*+[F]\txt{\rm{~Zinbiel~algebras}\\~$(A,\circ)$} \ar_{x\ast y=x\circ y+y\circ x}[d] \ar^{x\succ y=y\prec x=x\circ y}[r] & *+[F]\txt{\rm {dendriform~algebras}\\~$(A,\prec,\succ)$} \ar_{x\ast y=x\prec y+x\succ y}[d] \ar^{\quad x\star y=x\succ y-y\prec x}[r]   & *+[F]\txt{{\rm preLie~algebras}\\~$(A,\star)$} \ar_{[x,y]=x\star y-y\star x}[d]   \\
*+[F]\txt{\rm{commutative~algebras}\\~$(A,\ast)$} \ar^{\quad{\rm forget~commutativity}}[r]  & *+[F]\txt{\rm{associative~algebras}\\~$(A,\ast)$} \ar^{[x,y]=x\ast y-y\ast x}[r]  & *+[F]\txt{\rm{Lie~algebras}\\~$(A,[\cdot,\cdot])$}
}
}
\end{equation*}

\subsection{Zinbiel bialgebras}

A bialgebra theory for an algebraic structure usually contains two same algebraic structures on a vector space $A$ and its dual space $A^*$ such that certain compatibility conditions hold. Usually, certain matched pairs, Manin triples and bialgebras are equivalent. There are many important bialgebras, such as Lie bialgebras (\cite{Kosmann}), Hopf algebras (\cite{Mon}), and infinitesimal bialgebras (\cite{Aguiar3,Bai2,Zhe}). %Usually, there are certain subclasses of bialgebras, like quasi-triangular ones, triangular ones and factorizable ones.
In the context of Lie bialgebras, classical $r$-matrices give rise to  quasi-triangular Lie bialgebras (\cite{STS}). Moreover, factorizable Lie bialgebras, which  are special quasi-triangular Lie bialgebras,  have important applications in integrable systems (\cite{RS,S2}).

In this paper, we establish the bialgebra theory for Zinbiel algebras. We introduce the notions of matched pairs, Manin triples of Zinbiel algebras and Zinbiel bialgebras, and prove the equivalence among them. The main innovation is that we use skew-symmetric invariant bilinear form instead of symmetric invariant bilinear form in the definition of a quadratic Zinbiel algebras. Another ingredient is that the dual representation of a representation $(V;\rho,\mu)$ should be $(V^{\ast};-\rho^{\ast}-\mu^{\ast},\mu^{\ast})$ rather than $(V^{\ast};\mu^{\ast},\rho^{\ast})$.
We also introduce the notion of coboundary Zinbiel bialgebras, which leads to an analogue of the classical Yang-Baxter equation. Then we introduce the notions of quasi-triangular Zinbiel bialgebras and factorizable Zinbiel bialgebras. As an important example, we show that the Zinbiel double of a Zinbiel bialgebra naturally enjoys a factorizable Zinbiel bialgebra structure.
%Thus, our result is totally different from \cite{HD}, according to the different definitions of Manin triples of Zinbiel algebras and Zinbiel bialgebras. Actually, in \cite{HD}, the relation between Zinbiel algebras and Lie algebras is wrong and the equivalence between matched pairs, Manin triples of Zinbiel algebras and Zinbiel bialgebras had not been proved yet.

Note that the notion of dendriform D-bialgebras was introduced by Bai in the pioneer work \cite{Bai2}, which serves as the bialgebra for dendriform algebras. As aforementioned,  Zinbiel algebras can be viewed as special dendriform algebras. Thus one can obtain a bialgebra theory for Zinbiel algebras through the restriction of the results in \cite{Bai2} to the category of Zinbiel algebras. However, on the one hand, it is necessary to give explicit formulas for such structures due to the importance of Zinbiel algebras; On the other hand, there are differences in the coboundary case and associated analogue of the Yang-Baxter equation. In \cite{Bai2},   a strong condition $r_{\prec}=r, r_{\succ}=-\sigma(r)$ for some $2$-tensor $r$ is needed. So the analogue of the Yang-Baxter equation in this paper is different from the one given in \cite{Bai2}. Moreover, we also develop  quasi-triangular and factorizable theories for Zinbiel bialgebras, and these are totally new.

\subsection{Rota-Baxter characterization of factorizable Zinbiel bialgebras}

The notion of Rota-Baxter algebras was initiated by G. Baxter (\cite{Ba}) in his probability study to understand Spitzers identity in fluctuation theory. Rota-Baxter operators are broadly connected with mathematical physics, including the application of Connes-Kreimer's algebraic approach in the renormalization of perturbative quantum field theory (\cite{CK}). Rota-Baxter algebras are also closely related to double algebras (\cite{G3,Goncharov}), see \cite{Guo} for more details.  In the Lie algebra context, a Rota-Baxter operator was introduced independently in the 1980s as the operator form of the classical Yang-Baxter equation  (\cite{STS}). In order to better understand the relationship between the classical Yang-Baxter equation and the related integrable systems, Kupershmidt introduced an $\huaO$-operator on a Lie algebra (\cite{Ku}) as a more general notion.

Note that Rota-Baxter operators can be used to characterize certain bialgebras.  Goncharov established a correspondence between non-skewsymmetric solutions of the classical Yang-Baxter equation and  Rota-Baxter operators of nonzero weight for certain Lie algebras (\cite{G1,G2}). See \cite{BGN} for more general study. In \cite{Lang}, it showed that there is a one-to-one correspondence between factorizable Lie bialgebras and quadratic Rota-Baxter Lie algebras of nonzero weight. In \cite{SW}, factorizable antisymmetric infinitesimal bialgebras are characterized by symmetric Rota-Baxter Frobenius algebras.

In this paper, we also give the Rota-Baxter characterization of factorizable Zinbiel bialgebras. We introduce the notion of quadratic Rota-Baxter Zinbiel algebras of weight $\lambda$ and show that there is a one-to-one correspondence between factorizable Zinbiel bialgebras and quadratic Rota-Baxter Zinbiel algebras of nonzero weight. We also prove that quadratic Rota-Baxter Zinbiel algebras and Rota-Baxter commutative algebras with Connes cocycles are equivalent.

\subsection{Outline of the paper}

The paper is organized as follows. In Section \ref{Zinbiel-bialgebras}, we introduce the notions of matched pairs, Manin triples of Zinbiel algebras and Zinbiel bialgebras, and prove the equivalence among them. In Section \ref{sec:coboundary}, we study   coboundary Zinbiel bialgebras, which leads to an analogue of the classical Yang-Baxter equation. In Section \ref{sec:factorizable}, we introduce the notions of quasi-triangular and factorizable Zinbiel bialgebras. A quasi-triangular Zinbiel algebra can give rise to a relative Rota-Baxter operator of weight $-1$.  A factorizable Zinbiel bialgebra can give a factorization of the underlying Zinbiel algebra. We show that the Zinbiel double of a Zinbiel bialgebra naturally has a factorizable Zinbiel bialgebra structure. In Section \ref{sec:quadratic-RB}, we introduce the notion of quadratic Rota-Baxter Zinbiel algebras and show that there is a one-to-one correspondence between factorizable Zinbiel bialgebras and quadratic Rota-Baxter Zinbiel algebras. %We also prove that quadratic Rota-Baxter Zinbiel algebras and Rota-Baxter commutative (associative) algebras with Connes cocycles are equivalent.

Through out this paper, we work over a field $\mathbb{K}$ and all the vector spaces and algebras are over $\mathbb{K}$ and finite-dimensional.

\vspace{2mm}
\noindent
{\bf Acknowledgements.} We give our warmest thanks to Yunhe Sheng and Rong Tang for helpful discussions.
 This research is supported by NSFC (12471060,  W2412041).

\section{Quadratic Zinbiel algebras and Zinbiel bialgebras}\label{Zinbiel-bialgebras}

In this section, we introduce the notion of a matched pair of Zinbiel algebras. Then we introduce the notions of a Manin triple of Zinbiel algebras and a Zinbiel bialgebra. Finally, we prove that these objects are equivalent.

\subsection{Preliminaries and basic results}
First we recall the representations and some basic results of Zinbiel algebras.

\begin{defi}{\rm(\cite{Aguiar,Loday})}
A (left) {\bf Zinbiel algebra} $(A,\circ_A)$ is a vector space $A$ together with a bilinear operation $\circ_A:A \otimes A \to A$ such that
\begin{equation}\label{Zin-identity}
x\circ_A (y \circ_A z)=(x \circ_A y+y\circ_A x)\circ_A z, \quad \forall x,y,z \in A,
\end{equation}
or equivalently,
$$  (x,y,z)=(y\circ_A x)\circ_A z, \quad \forall x,y,z \in A, $$
where $(x,y,z)=x\circ_A (y \circ_A z)-(x \circ_A y)\circ_A z$ is the associator.
\end{defi}

\emptycomment{
\begin{rmk}
We can also define a right Zinbiel algebra $(A,\bullet_A)$ satisfying the following identity:
$$ (x\bullet_A y) \bullet_A z=x \bullet_A (y \bullet_A z+z \bullet_A y), \quad \forall x,y,z \in A.  $$
Obviously, the opposite algebra of a left Zinbiel algebra is the right Zinbiel algebra under the same vector space. Actually, right Zinbiel algebras (also called dual Leibniz algebras) are koszul dual to the (left) Leibniz algebras. More details can be seen in \cite{Balavoine,Loday}.
\end{rmk}
}

Let $(A,\circ_A)$ be a Zinbiel algebra. Then the bilinear operation $\ast_A:A \otimes A \to A$ defined by
\begin{equation}\label{Zin-com}
x \ast_A y=x \circ_A y+y\circ_A x, \quad \forall  x,y \in A
\end{equation}
defines a commutative associative algebra $(A,\ast_A)$, which is called the {\bf sub-adjacent commutative associative algebra} of $(A,\circ_A)$ and denoted by $A^c$. Recall that a representation of a commutative associative algebra $(A,\ast_A)$ is a pair $(V,\zeta)$, where $V$ is a vector space and $\zeta:A \to \End(V)$ is a linear map satisfying:
\begin{equation}
\zeta(x\ast_A y)=\zeta(x)\zeta(y),\quad \forall x,y\in A.
\end{equation}
In this situation, $L:A \to \End(A)$ defined by $L_x(y)=x\circ_A y$ gives a representation of the commutative associative algebra $A^c$ on $A$.

\begin{pro}
Every Zinbiel algebra $(A,\circ_A)$ satisfies the equation:
\begin{equation}\label{left-com}
x\circ_A (y \circ_A z)=y\circ_A (x\circ_A z), \quad \forall x,y,z \in A.
\end{equation}
\end{pro}

\begin{proof}
It can be deduced directly from the equation \eqref{Zin-identity}.
\end{proof}

\emptycomment{
\begin{pro}
Let $(A,\prec,\succ)$ be a dendriform algebra. Then $(A,\circ)$ is a Zinbiel algebra, where $\circ$ is defined by
$$ x\circ y=x \prec y=y \succ x, \quad \forall x,y \in A.$$
\end{pro}

\begin{proof}
It can be obtained obviously from the definition of dendriform algebras.
\end{proof}
}

\begin{defi}
A {\bf representation} of a Zinbiel algebra $(A,\circ_A)$ is a triple $(V;\rho,\mu)$, where $V$ is a vector space and $\rho,\mu:A \to \End(V)$ are linear maps such that the following equations hold for all $x,y \in A$
\begin{eqnarray}
\label{rep1}&&\rho(x)\rho(y)=\rho(x\circ_A y)+\rho(y\circ_A x);\\
\label{rep2}&&\rho(x)\mu(y)=\mu(x\circ_A y)=\mu(y)\rho(x)+\mu(y)\mu(x);
\end{eqnarray}
\end{defi}

In fact, $(V;\rho,\mu)$ is a representation of a Zinbiel algebra $(A,\circ_A)$ if and only if the direct sum $A\oplus V$ of vector spaces is a Zinbiel algebra (the semi-direct product) by defining the multiplication $\circ_{(\rho,\mu)}$ on $A\oplus V$ by
$$ (x+u)\circ_{(\rho,\mu)}(y+v)=x \circ_A y+\rho(x)v+\mu(y)u, \quad \forall x,y\in A,u,v\in V.$$
We denote it by $A\ltimes_{\rho,\mu} V$ or simply by $A\ltimes V$.

\begin{ex}{\rm
Let $(A,\circ_A)$ be a Zinbiel algebra.  Then $(A;L,R)$ is a representation of $(A,\circ_A)$, which is called the {\bf regular representation}, where $R:A\to \End(A)$ is defined by $R_{x}(y):=y\circ_A x$, for all $x,y\in A$. If there is a Zinbiel algebra structure on the dual space $A^{\ast}$, we denote the left and right multiplications by $\huaL$ and $\huaR$ respectively.
}
\end{ex}

Now we investigate the dual representation in the Zinbiel algebra context, in order to relate the matched pairs of Zinbiel algebras to Zinbiel bialgebras and Manin triples of Zinbiel algebras later.

\begin{pro}\label{dual-rep}
Let $(V;\rho,\mu)$ be a representation of a Zinbiel algebra $(A,\circ_A)$. Then
$$(V^{\ast};-\rho^{\ast}-\mu^{\ast},\mu^{\ast})$$
is a representation of $(A,\circ_A)$, which is called the {\bf dual representation} of $(V;\rho,\mu)$.
\end{pro}

\begin{proof}
By \eqref{rep1} and \eqref{rep2}, for all $x,y,z\in A$ and $\xi\in A^*$, we have
\begin{eqnarray*}
&&\langle \Big((-\rho^*-\mu^*)(x)(-\rho^*-\mu^*)(y)-(-\rho^*-\mu^*)(x\circ_A y+y\circ_A x)\Big)\xi,z \rangle\\
&=&\langle \xi,\rho(y)\rho(x)z+ \mu(y)\rho(x)z+\rho(y)\mu(x)z+\mu(y)\mu(x)z   \rangle\\
&&-\langle \xi, \rho(x \circ_A y+y\circ_A x)z   \rangle-\langle \xi, \mu(x \circ_A y+y\circ_A x)z   \rangle\\
&=&0,
\end{eqnarray*}
which implies that $(-\rho^*-\mu^*)(x)(-\rho^*-\mu^*)(y)=(-\rho^*-\mu^*)(x\circ_A y+y\circ_A x)$. Similarly, we can show that
\begin{eqnarray*}
&&\langle \Big((-\rho^*-\mu^*)(x)\mu^*(y)-\mu^*(y)(-\rho^*-\mu^*)(x)-\mu^*(y)\mu^*(x)\Big)\xi,z\rangle\\
&=& \langle \xi,-\mu(y)\rho(x)z-\mu(y)\mu(x)z+\rho(x)\mu(y)z+\mu(x)\mu(y)z-\mu(x)\mu(y)z\rangle\\
&=& 0,
\end{eqnarray*}
and
\begin{eqnarray*}
&&\langle \Big(\mu^*(x\circ_A y)-\mu^*(y)(-\rho^*-\mu^*)(x)-\mu^*(y)\mu^*(x)\Big)\xi,z\rangle\\
&=& \langle \xi,-\mu(x\circ_A y)z+\rho(x)\mu(y)z+\mu(x)\mu(y)z-\mu(x)\mu(y)z\rangle\\
&=& 0.
\end{eqnarray*}
Thus, $(V^{\ast};-\rho^{\ast}-\mu^{\ast},\mu^{\ast})$ is a representation of $(A,\circ_A)$.
\end{proof}

\begin{ex}{\rm
Let $(A,\circ_A)$ be a Zinbiel algebra. By proposition \ref{dual-rep}, $(A^{\ast};-L^{\ast}-R^{\ast},R^{\ast})$ is a representation of $(A,\circ_A)$, which is called the {\bf coregular representation}, where two linear maps $L^{\ast},R^{\ast}:A\to \End(A^{\ast})$ with $x\to L^{\ast}_x $ and $x\to R^{\ast}_x $ are respectively defined by
$$ \langle L^{\ast}_x (\xi),y \rangle=-\langle \xi, x\circ_A y \rangle, \quad \langle R^{\ast}_x (\xi),y \rangle=-\langle \xi, y\circ_A x \rangle, \quad  \forall x,y\in A, \xi \in A^{\ast}. $$
}
\end{ex}

\subsection{Matched pairs, Manin triples of Zinbiel algebras and Zinbiel bialgebras}

\begin{defi}
Let $(A,\circ_A)$ and $(B,\circ_B)$ be two Zinbiel algebras. If there exists a representation $(\rho,\mu)$ of A on B and a representation $(\rho',\mu')$ of B on A satisfying the following identities:
\begin{eqnarray}
\label{mp-zinbiel-1}u\circ_B(\mu(x)v)+\mu(\rho'(v)x)u&=&\mu(x)(u\circ_B v+v\circ_B u);\\
\label{mp-zinbiel-2}\rho(x)(u \circ_B v)-(\rho(x)u)\circ_B v&=&\rho(\mu'(u)x)v+(\mu(x)u)\circ_B v+\rho(\rho'(u)x)v;\\
\label{mp-zinbiel-3}u\circ_B(\rho(x)v)+\mu(\mu'(v)x)u-(\mu(x)u)\circ_B v&=&(\rho(x)u)\circ_B v+\rho(\mu'(u)x)v+\rho(\rho'(u)x)v;\\
\label{mp-zinbiel-4}x\circ_A(\mu'(u)y)+\mu'(\rho(y)u)x&=&\mu'(u)(x\circ_A y+y\circ_A x);\\
\label{mp-zinbiel-5}\rho'(u)(x \circ_A y)-(\rho'(u)x)\circ_A y&=&\rho'(\mu(x)u)y+(\mu'(u)x)\circ_A y+\rho'(\rho(x)u)y;\\
\label{mp-zinbiel-6}x\circ_A(\rho'(u)y)+\mu'(\mu(y)u)x-(\mu'(u)x)\circ_A y&=&(\rho'(u)x)\circ_A y+\rho'(\mu(x)u)y+\rho'(\rho(x)u)y,
\end{eqnarray}
for all $x,y\in A$ and $u,v\in B$. Then we call $(A,B;(\rho,\mu),(\rho',\mu'))$ a {\bf matched pair of Zinbiel algebras}.
\end{defi}

\begin{pro}\label{mp-equ}
Let $(A,B;(\rho,\mu),(\rho',\mu'))$ be a matched pair of Zinbiel algebras. Then there exists a Zinbiel algebra structure on $A\oplus B$ defined by
\begin{equation}\label{mp-zinbiel}
(x+u)\circ_{\bowtie}(y+v)=x\circ_A y+\rho'(u)y+\mu'(v)x+u\circ_B v+\rho(x)v+\mu(y)u.
\end{equation}
Denote this Zinbiel algebra by $A \bowtie_{\rho',\mu'}^{~\rho,\mu} B,$ or simply by
$A \bowtie B.$

Conversely, if $(A \oplus B,\circ_{\bowtie})$ is a Zinbiel algebra such
that $A$ and $B$ are Zinbiel subalgebras, then
$(A,B;(\rho,\mu),(\rho',\mu'))$ is a matched pair of Zinbiel
algebras, where the representation $(\rho,\mu)$ of $A$ on $B$
and the representation $(\rho',\mu')$ of $B$ on $A$ are
determined  by
$$ x\circ_{\bowtie} u=\rho(x)u+\mu'(u)x,\quad u \circ_{\bowtie} x=\mu(x)u+\rho'(u)x,\quad  \forall x\in A, u\in B. $$
\end{pro}

\begin{proof}
For all $x,y,z\in A$ and $u,v,w \in B$, the equation \eqref{mp-zinbiel} defines a Zinbiel algebra structure on $A\oplus B$ if and only if the following equations are satisfied:
\begin{eqnarray*}
x \circ_{\bowtie} (y \circ_{\bowtie} z)=(x\circ_{\bowtie}y+y \circ_{\bowtie}x)\circ_{\bowtie} z &\Longleftrightarrow& {\rm (A,\circ_A)~is~a~Zinbiel~algebra;}\\
x \circ_{\bowtie} (y \circ_{\bowtie} w)=(x\circ_{\bowtie}y+y \circ_{\bowtie}x)\circ_{\bowtie} w &\Longleftrightarrow& {\rm the~equations~\eqref{rep1}~and~\eqref{mp-zinbiel-4}~hold;}\\
x \circ_{\bowtie} (v \circ_{\bowtie} z)=(x\circ_{\bowtie}v+v \circ_{\bowtie}x)\circ_{\bowtie} z &\Longleftrightarrow& {\rm the~equations~\eqref{rep2}~and~\eqref{mp-zinbiel-6}~hold;}\\
u \circ_{\bowtie} (y \circ_{\bowtie} z)=(u\circ_{\bowtie}y+y \circ_{\bowtie}u)\circ_{\bowtie} z &\Longleftrightarrow& {\rm the~equations~\eqref{rep2}~and~\eqref{mp-zinbiel-5}~hold;}\\
u \circ_{\bowtie} (v \circ_{\bowtie} w)=(u\circ_{\bowtie}v+v \circ_{\bowtie}u)\circ_{\bowtie} w &\Longleftrightarrow& {\rm (B,\circ_B)~is~a~Zinbiel~algebra;}\\
u \circ_{\bowtie} (v \circ_{\bowtie} z)=(u\circ_{\bowtie}v+v \circ_{\bowtie}u)\circ_{\bowtie} z &\Longleftrightarrow& {\rm the~equations~\eqref{rep1}~and~\eqref{mp-zinbiel-1}~hold;}\\
u \circ_{\bowtie} (y \circ_{\bowtie} w)=(u\circ_{\bowtie}y+y \circ_{\bowtie}u)\circ_{\bowtie} w &\Longleftrightarrow& {\rm the~equations~\eqref{rep2}~and~\eqref{mp-zinbiel-3}~hold;}\\
x \circ_{\bowtie} (v \circ_{\bowtie} w)=(x\circ_{\bowtie}v+v \circ_{\bowtie}x)\circ_{\bowtie} w &\Longleftrightarrow& {\rm the~equations~\eqref{rep2}~and~\eqref{mp-zinbiel-2}~hold.}
\end{eqnarray*}
Thus, $(A,B;(\rho,\mu),(\rho',\mu'))$ is a matched pair of Zinbiel algebras if and only if  $(B;\rho,\mu)$ is a representation of $A$ and $(A;\rho',\mu')$ is a representation of $B$ and they satisfy the equations \eqref{mp-zinbiel-1}-\eqref{mp-zinbiel-6}.
\end{proof}

Let us recall the notion of a matched pair of commutative associative algebras.
\begin{defi}{\rm(\cite{NB})}
Let $(A,\ast_1)$ and $(B,\ast_2)$ be two commutative associative algebras. If there exists a representation $(B,\zeta)$ of $A$ and a representation $(A,\zeta')$ of $B$ satisfying the following equations:
\begin{eqnarray}
\label{mp-com-1}\zeta(x)(u\ast_2 v)&=&(\zeta(x)u)\ast_2 v+\zeta(\zeta'(u)x)v;\\
\label{mp-com-2}\zeta'(u)(x\ast_1 y)&=&(\zeta'(u)x)\ast_1 y+\zeta'(\zeta(x)u)y,
\end{eqnarray}
for all $x,y \in A$ and $u,v \in B$. Then $(A,B;\zeta,\zeta')$ is called a {\bf matched pair of commutative associative algebras}.
\end{defi}

In this case, there exists a commutative associative algebra structure $\ast_{\bowtie}$ on the vector space $A\oplus B$ given by
\begin{equation}
(x+u)\ast_{\bowtie}(y+v)=x\ast_1 y+\zeta'(u)y+\zeta'(v)x+u\ast_2 v+\zeta(x)v+\zeta(y)u.
\end{equation}

It can be denoted by $ A\bowtie^{\zeta'}_{\zeta} B $ or simply $A\bowtie B$. Moreover, every commutative associative algebra which is the direct sum of the underlying vector spaces of two subalgebras can be obtained from a matched pair of commutative associative algebras.

\begin{pro}\label{mpZin-mpcom}
Let $(A,B;(\rho,\mu),(\rho',\mu'))$ be a matched pair of Zinbiel algebras. Then $(A^c,B^c;\rho+\mu,\rho'+\mu')$ be a matched pair of commutative associative algebras.
\end{pro}

\begin{proof}
This inclusion can be proved by a direct calculation or from the relation between the Zinbiel algebra $A\bowtie B$ and its sub-adjacent commutative associative algebra. In fact, the sub-adjacent commutative associative algebra $(A\bowtie B)^c$ is just the commutative associative algebra $A^c \bowtie^{\zeta'}_{\zeta} B^c $  obtained from the matched pair $(A^c,B^c;\zeta,\zeta')$:
$$ (x+u)\ast_{\bowtie}(y+v)=x\ast_A y+\zeta'(u)y+\zeta'(v)x+u\ast_B v+\zeta(x)v+\zeta(y)u, \quad \forall x,y\in A^c,u,v\in B^c,$$
where $\zeta=\rho+\mu$ and $\zeta'=\rho'+\mu'.$
\end{proof}

\begin{thm}\label{equ-mps}
Let $(A,\circ_A)$ and $(A^*,\cdot_{A^*})$ be two Zinbiel algebras. then $(A,A^*;-L^*-R^*,R^*,-\huaL^*-\huaR^*,\huaR^*)$ is a matched pair of Zinbiel algebras if and only if  $(A^c,A^{*c};-L^*,-\huaL^*)$ is a matched pair of commutative associative algebras.
\end{thm}

\begin{proof}
By Proposition \ref{mpZin-mpcom}, we can obviously know that the ``only if" part is right.
We only need to prove that the ``if" part. A direct proof is given as follows. In the case $\rho=-L^*-R^*,\mu=R^*,\rho'=-\huaL^*-\huaR^*,\mu'=\huaR^*$ and $\zeta=-L^*,\zeta'=-\huaL^*$, we have
\begin{eqnarray*}
{\rm equation~\eqref{mp-zinbiel-1}}\Longleftrightarrow {\rm equation~\eqref{mp-zinbiel-5}}&\Longleftrightarrow&{\rm equation~\eqref{mp-zinbiel-6}}\Longleftrightarrow{\rm equation~\eqref{mp-com-1}}\\
{\rm equation~\eqref{mp-zinbiel-2}}\Longleftrightarrow {\rm equation~\eqref{mp-zinbiel-3}}&\Longleftrightarrow&{\rm equation~\eqref{mp-zinbiel-4}}\Longleftrightarrow{\rm equation~\eqref{mp-com-2}}.
\end{eqnarray*}
As an example, we show that how the equation \eqref{mp-zinbiel-1} is equivalent to the equation  \eqref{mp-com-1}. In fact, it follows from
\begin{eqnarray*}
\langle R_x^*(u\ast_B v),y\rangle&=&\langle L_y^*(u\ast_B v),x\rangle;\\
\langle R^* \big(-\huaL_v^*(x)-\huaR_v^*(x) \big)u,y\rangle&=&\langle (L_y^* u)\ast_B v,x\rangle;\\
\langle u\circ_B (R_x^* v),y\rangle&=&-\langle L^*(\huaL^*_u y)v,x\rangle,
\end{eqnarray*}
for all $x,y\in A$ and $u,v\in B$. Actually, the proof of other equivalent relations is similar. It can be proved in the same way and we omit it.
\end{proof}

Now we introduce the notion of a quadratic Zinbiel algebra,which is an important ingredient in the study of Zinbiel bialgebras.

\begin{defi}\label{quadratic-Zinbiel-alg}
A {\bf quadratic Zinbiel algebra} is a Zinbiel algebra $(A,\circ_A,\omega)$ equipped with a nondegenerate skew-symmetric bilinear form $\omega\in \wedge^2 A^*$ such that the following invariant condition holds:
\begin{equation}\label{quadratic-condition1}
\omega(x\circ_A y,z)=\omega(y,x\circ_A z+z\circ_A x),\quad \forall x,y,z\in A.
\end{equation}
\end{defi}

\begin{rmk}
In fact, by \eqref{quadratic-condition1}, we have
\begin{equation}\label{quadratic-condition2}
\omega(x\circ_A y,z)=\omega(y,x\circ_A z+z\circ_A x)=\omega(z\circ_A y,x),\quad \forall x,y,z\in A.
\end{equation}
\end{rmk}

\begin{defi}
A {\bf Manin triple of Zinbiel algebras} is a triple $(\huaA,A_1,A_2)$, where
\begin{itemize}
\item[{\rm(i)}] $(\huaA,\circ_{\huaA},\omega)$ is a quadratic Zinbiel algebra.

\item[{\rm(ii)}] both $A_1$ and $A_2$ are isotropic subalgebras of $(\huaA,\circ_{\huaA})$ with respect to $\omega$.

\item[{\rm(iii)}]$\huaA=A_1\oplus A_2$ as vector spaces.
\end{itemize}
\end{defi}

\begin{ex}{\rm
Let $(A,\circ_A)$ be a Zinbiel algebra. Then $(A\ltimes_{-L^*-R^*,R^*} A^*,A,A^*)$ is a Manin triple of Zinbiel algebras, where the natural nondegenerate skew-symmetric bilinear form $\omega$ on $A\oplus A^*$ is given by
\begin{equation}\label{natural-bilinear-form}
\omega(x+\xi,y+\eta)=\langle\xi,y\rangle-\langle \eta,x \rangle, \quad \forall x,y\in A, \xi,\eta\in A^*.
\end{equation}
}
\end{ex}

\begin{defi}
Let $A$ be a vector space. A {\bf Zinbiel bialgebra} structure on $A$ is a pair of linear maps $(\alpha,\beta)$ such that $\alpha:A\to A\otimes A,\beta:A^*\to A^*\otimes A^*$ and
\begin{itemize}
\item[{\rm(i)}] $\alpha^*: A^*\otimes A^* \to A^*$ is a Zinbiel algebra structure on $A^*$.

\item[{\rm(ii)}] $\beta^*: A\otimes A \to A$ is a Zinbiel algebra structure on $A$.

\item[{\rm(iii)}] $\alpha$ is a 1-cocycle of the sub-adjacent commutative associative algebra $A^c$ with coefficients in the representation $\Big(A\otimes A;L\otimes \Id+\Id \otimes (L+R)\Big)$.

\item[{\rm(iv)}] $\beta$ is a 1-cocycle of the sub-adjacent commutative associative algebra $A^{\ast c}$ with coefficients in the representation $\Big(A^*\otimes A^*;\huaL\otimes \Id+\Id \otimes (\huaL+\huaR)\Big)$.
\end{itemize}
  We denote this {\bf Zinbiel bialgebra} by $(A,A^*,\alpha,\beta)$, or simply by $(A,A^*)$.
\end{defi}

Up to now, we have introduced the notions of a matched pair and a Manin triple of Zinbiel algebras and a Zinbiel bialgebra. Next we show that these objects are equivalent when we consider the coregular representation in a matched pair of Zinbiel algebras. The following theorem is the main result in this section.

\begin{thm}
Let $(A,\circ_A)$ and $(A^*,\cdot_{A^*})$ be two Zinbiel algebras. Then the following conditions are equivalent:
\begin{itemize}
\item[{\rm(i)}] $(A,A^*)$ is a Zinbiel bialgebra.

\item[{\rm(ii)}] $(A,A^*;-L^*-R^*,R^*,-\huaL^*-\huaR^*,\huaR^*)$ is a matched pair of Zinbiel algebras.

\item[{\rm(iii)}] $(A\oplus A^*,A,A^*)$ is a Manin triple of Zinbiel algebras, where the invariant skew-symmetric bilinear form $\omega$ on $A\oplus A^*$ is given by \eqref{natural-bilinear-form}.
\end{itemize}
\end{thm}

\begin{proof}
First we can prove that {\rm(ii)} is equivalent to {\rm(iii)}.

Let $(A,A^*;-L^*-R^*,R^*,-\huaL^*-\huaR^*,\huaR^*)$ be a matched pair of Zinbiel algebras. Then $(A\oplus A^*,\circ_{\bowtie})$ is a Zinbiel algebra, where $\circ_{\bowtie}$ is given by \eqref{mp-zinbiel}. We only need to prove that $\omega$ satisfies the invariant condition \eqref{quadratic-condition1}. For all $x,y,z\in A$ and $\xi,\eta,\theta\in A^*$, we have
\begin{eqnarray*}
&&\omega\Big( (x+\xi)\circ_{\bowtie}(y+\eta), z+\theta   \Big)\\
&=& \omega\Big( x\circ_A y-(\huaL_{\xi}^*+\huaR_{\xi}^*)y+\huaR_{\eta}^* x+\xi \cdot_{A^*} \eta-
(L_x^*+R_x^*)\eta+R_y^*\xi,z+\theta \Big)\\
&=& \langle \xi \cdot_{A^*} \eta-(L_x^*+R_x^*)\eta+R_y^*\xi,z  \rangle-\langle\theta,x\circ_A y-(\huaL_{\xi}^*+\huaR_{\xi}^*)y \rangle\\
&=& \langle \xi \cdot_{A^*} \eta,z  \rangle+\langle \eta,x\circ_A z+z\circ_A x \rangle-\langle \xi,z\circ_A y\rangle-\langle \theta,x\circ_A y \rangle\\
&&-\langle \xi \cdot_{A^*} \theta+ \theta \cdot_{A^*} \xi,y \rangle+\langle \theta\cdot_{A^*} \eta,x \rangle.
\end{eqnarray*}
Similarly, we also have
\begin{eqnarray*}
&&\omega\Big(y+\eta,(x+\xi)\circ_{\bowtie}(z+\theta)+(z+\theta)\circ_{\bowtie}(x+\xi) \Big)\\
&=& \omega \Big(y+\eta,x\circ_A z+z\circ_A x-\huaL_{\xi}^* z-\huaL_{\theta}^* x+\theta \cdot_{A^*} \xi+\xi \cdot_{A^*} \theta-L_x^* \theta-L_z^* \xi \Big)\\
&=& \langle \eta, x\circ_A z+z\circ_A x-\huaL_{\xi}^* z-\huaL_{\theta}^* x \rangle-\langle \theta \cdot_{A^*} \xi+\xi \cdot_{A^*} \theta-L_x^* \theta-L_z^* \xi, y  \rangle\\
&=& \langle \eta,x\circ_A z+z\circ_A x \rangle+\langle \xi \cdot_{A^*} \eta,z\rangle+\langle \theta\cdot_{A^*} \eta,x \rangle-\langle \xi \cdot_{A^*} \theta+ \theta \cdot_{A^*} \xi,y\rangle\\
&&-\langle \theta,x\circ_A y \rangle-\langle \xi,z\circ_A y\rangle.
\end{eqnarray*}
Thus, $\omega$ satisfies the invariant condition \eqref{quadratic-condition1}.

On the other hand, if $(\huaA,A,A^*)$ is a Manin triple of Zinbiel algebras with the invariant bilinear form given by \eqref{natural-bilinear-form}, for all $x\in A,\xi,\eta\in A^*$, by \eqref{quadratic-condition2}, we have
$$\langle -\mu'(\xi)x,\eta\rangle=\omega(x\circ_{\bowtie} \xi, \eta)=-\omega(x,\eta \circ_{\bowtie} \xi)=\langle  \eta \cdot_{A^*} \xi,x \rangle=\langle\eta, -\huaR_{\xi}^* x \rangle,$$
which implies that $\mu'=\huaR^*.$ By \eqref{quadratic-condition1}, we also have
$$\langle\eta,(\rho'+\mu')(\xi)x\rangle=\omega(\eta,x\circ_{\bowtie} \xi+\xi \circ_{\bowtie} x)= \omega(x\circ_{\bowtie} \eta,\xi)= \omega(\huaR_{\eta}^*x,\xi)=\langle \eta,-\huaL_{\xi}^*x \rangle,   $$
which implies that $\rho'=-\huaL^*-\huaR^*.$
Using the same way, we can also have $\mu=R^*$ and $\rho=-L^*-R^*$. Thus, $(A,A^*;-L^*-R^*,R^*,-\huaL^*-\huaR^*,\huaR^*)$ is a matched pair of Zinbiel algebras.

Next we prove that {\rm(i)} is equivalent to {\rm(ii)}.

Let $(A,\circ_A)$ and $(A^*,\cdot_{A^*})$ be two Zinbiel algebras. If $(A,A^*;-L^*-R^*,R^*,-\huaL^*-\huaR^*,\huaR^*)$ is a matched pair of Zinbiel algebras, by Theorem \ref{equ-mps}, equivalently, $(A^c,A^{*c};-L^*,-\huaL^*)$ is a matched pair of commutative associative algebras. Consider the equation \eqref{mp-com-1}, we have
\begin{eqnarray*}
&&\langle -L_x^*(\xi\ast_{A^*} \eta)+(L_x^*\xi)\cdot_{A^*} \eta+ \eta \cdot_{A^*}(L_x^*\xi)-L^*(\huaL_{\xi}^* x)\eta, y \rangle\\
&=&\langle\xi\ast_{A^*} \eta, x\circ_A y
\rangle+\langle \xi,x\circ_A (\huaR_{\eta}^* y) \rangle+\langle \xi,x \circ_A (\huaL_{\eta}^* y) \rangle+\langle \eta, (\huaL_{\xi}^* x)\circ_A y  \rangle\\
&=&\langle \beta(\xi\ast_{A^*} \eta)-(\Id \otimes \huaL_{\eta}+\huaR_{\eta})\beta(\xi)-(\huaL_{\xi}\otimes \Id)\beta(\eta), x\otimes y\rangle\\
&=& 0,
\end{eqnarray*}
which implies that $\beta(\xi\ast_{A^*} \eta)=(\huaL_{\xi}\otimes \Id)\beta(\eta)+(\Id \otimes (\huaL_{\eta}+\huaR_{\eta}))\beta(\xi)$. So $\beta$ is a 1-cocycle of $A^{\ast c}$ with coefficients in the representation $\Big(A^*\otimes A^*;\huaL\otimes \Id+\Id \otimes (\huaL+\huaR)\Big)$.
Similarly, by \eqref{mp-com-2}, we also have
\begin{eqnarray*}
&&\langle -\huaL_{\xi}^*(x\ast_A y)+(\huaL_{\xi}^*x)\ast_A y+y \ast_A (\huaL_{\xi}^*x)-\huaL(L_x^* \xi)y, \eta \rangle\\
&=&\langle x\ast_A y,\xi\cdot_{A^*} \eta\rangle+\langle x,\xi\cdot_{A^*} (R^*_y \eta) \rangle+\langle x,\xi\cdot_{A^*} (L^*_y \eta)\rangle+\langle y,(L_x^* \xi)\cdot_{A^*}\eta \rangle\\
&=&\langle \alpha(x\ast_A y)-(L_x\otimes \Id)\alpha(y)-(\Id\otimes L_y+R_y)\alpha(x),\xi\otimes \eta\rangle\\
&=& 0,
\end{eqnarray*}
which implies that $\alpha(x\ast_A y)=(L_x\otimes \Id)\alpha(y)+(\Id\otimes L_y+R_y)\alpha(x)$.
So $\alpha$ is a 1-cocycle of $A^c$ with coefficients in the representation $\Big(A\otimes A;L\otimes \Id+\Id \otimes (L+R)\Big)$.
Thus $(A,A^*)$ is a Zinbiel bialgebra if and only if $(A^c,A^{*c};-L^*,-\huaL^*)$ is a matched pair of commutative associative algebras, or equivalently, $(A,A^*;-L^*-R^*,R^*,-\huaL^*-\huaR^*,\huaR^*)$ is a matched pair of Zinbiel algebras. The proof is finished.
\end{proof}

\begin{defi}
Let $(A,A^*,\alpha_A,\beta_A)$ and $(B,B^*,\alpha_B,\beta_B)$ be two Zinbiel bialgebras. A linear map $\varphi:A \to B$ is called {\bf a homomorphism of Zinbiel bialgebras}, if $\varphi:A \to B$ is a homomorphism of Zinbiel algebras such that $\varphi^*:B^* \to A^*$ is also a homomorphism of Zinbiel algebras, or equivalently, $\varphi$ satisfying
\begin{eqnarray}
(\varphi \otimes \varphi)\circ \alpha_A=\alpha_B \circ \varphi, \quad (\varphi^* \otimes \varphi^*)\circ \beta_B=\beta_A \circ \varphi^*.
\end{eqnarray}
Furthermore, if $\varphi:A \to B$ is a linear isomorphism of vector spaces, then $\varphi$ is called {\bf an isomorphism of Zinbiel bialgebras.}
\end{defi}

\begin{pro}\label{pro:induced isomorphism}
Let $(A,A^*,\alpha_A,\beta_A)$ be a Zinbiel bialgebra and $B$ a vector space. Assume that $\varphi:A \to B$ is a linear isomorphism between vector spaces. Define product $\circ_B:B\otimes B \to B$ and $\cdot_{B^*}:B^*\otimes B^* \to B^*$ respectively by
$$  u\circ_B v=\varphi(\varphi^{-1}(u) \circ_A  \varphi^{-1}(v)),\quad \xi' \cdot_{B^*} \eta'=(\varphi^*)^{-1} (\varphi^*(\xi')\cdot_{A^*} \varphi^*(\eta')),\quad \forall u,v\in B,\xi',\eta' \in B^*.$$
Then $(B,B^*,\alpha_B,\beta_B)$ is a Zinbiel bialgebra, where $\alpha_B:B \to B\otimes B$ and $\beta_B:B^* \to B^* \otimes B^*$ are given respectively by
$$ \langle \alpha'(u),\xi'\otimes \eta' \rangle=\langle u,\xi' \cdot_{B^*} \eta' \rangle,\quad
\langle \beta'(\xi'),u \otimes v \rangle=\langle \xi', u \circ_{B} v \rangle, \quad \forall u,v\in B,\xi',\eta' \in B^*. $$
Moreover, $\varphi$ is a Zinbiel bialgebra isomorphism between $(A,A^*,\alpha_A,\beta_A)$ and $(B,B^*,\alpha_B,\beta_B)$.
\end{pro}

\begin{proof}
It follows by a direct calculation. We omit the details.
\end{proof}

\begin{ex}\label{dual Zinbiel alg}
{\rm
Let $(A,A^*,\alpha,\beta)$ be a Zinbiel bialgebra. Then its dual $(A^*,A,\beta,\alpha)$ is also a Zinbiel bialgebra.
}
\end{ex}

\begin{ex}{\rm
Let $(A,\circ_A)$ be a Zinbiel algebra and the Zinbiel algebra structure on $A^*$ be trivial, then in this case $(A,A^*,0,\beta)$ is a Zinbiel bialgebra, which corresponds to the Zinbiel algebra $A \ltimes_{-L^*-R^*,R^*} A^*.$
}
\end{ex}

\section{Coboundary Zinbiel bialgebras and the classical Zinbiel Yang-Baxter equation}\label{sec:coboundary}

In this section, we study the case that the linear map $\alpha$ is a 1-coboundary of the sub-adjacent commutative associative algebra $A^c$ with coefficients in the representation $\Big(A\otimes A;L\otimes \Id+\Id \otimes (L+R)\Big)$. Then we obtain a equivalent characterization of a coboundary Zinbiel bialgebra, which gives rise to the Zinbiel Yang-Baxter equation, as an analogue of the classical Yang-Baxter equation.

\begin{defi}
A Zinbiel bialgebra $(A,A^*,\alpha,\beta)$ is called {\bf coboundary} if $\alpha$ is a 1-coboundary of $A^c$ with coefficients in the representation $\Big(A\otimes A;L\otimes \Id+\Id \otimes (L+R)\Big)$. It means, there exists a $r\in A\otimes A$ such that
\begin{equation}\label{coboundary-condition}
\alpha(x)=\Big( \Id\otimes (L_x+R_x)-L_x \otimes \Id  \Big)r, \quad \forall x\in A.
\end{equation}
\end{defi}

Let $A$ be a Zinbiel algebra whose product is given by $\beta^*:A\otimes A\to A$ and $r\in A\otimes A$. Suppose that $\alpha$ satisfies the equation \eqref{coboundary-condition}, then
it's straightforward to check that $\alpha$ is a 1-cocycle of $A^c$ with coefficients in the representation $\Big(A\otimes A;L\otimes \Id+\Id \otimes (L+R)\Big)$ by direct calculations. Therefore, $(A,A^*,\alpha,\beta)$ is a Zinbiel bialgebra if and only if the following conditions hold:
\begin{itemize}
\item[{\rm(i)}] $\alpha^*: A^*\otimes A^* \to A^*$ is a Zinbiel algebra structure on $A^*$.

\item[{\rm(ii)}] $\beta$ is a 1-cocycle of $A^{\ast c}$ with coefficients in the representation $\Big(A^*\otimes A^*;\huaL\otimes \Id+\Id \otimes (\huaL+\huaR)\Big)$, where the Zinbiel algebra structure on $A^*$ is given by {\rm(i)}.
\end{itemize}

\begin{pro}
Let $(A,\circ_A)$ be a Zinbiel algebra whose product is given by $\beta^*:A\otimes A\to A$ and $r\in A\otimes A$. Suppose there exists a Zinbiel algebra structure $\cdot_{A^*}$ given by $\alpha^*:A^*\otimes A^*\to A^*$, where $\alpha$ is given by \eqref{coboundary-condition}. Then $\beta:A^* \to A^*\otimes A^*$ is a 1-cocycle of $A^{*c}$ with coefficients in the representation $\Big(A\otimes A;L\otimes \Id+\Id \otimes (L+R)\Big)$ if and only if $r$ satisfies
\begin{equation}\label{beta-1-cocycle}
\Big( L_{x\circ_A y}\otimes \Id-\Id \otimes L_{x\circ_A y}-L_x L_y\otimes \Id+L_x\otimes L_y \Big)(r-\sigma r)=0,
\end{equation}
where the linear map $\sigma:A\otimes A \to A\otimes A$ is a exchanging operator satisfying $\sigma(x\otimes y)=y\otimes x$, for all $x,y\in A.$

\end{pro}

\begin{proof}
Since $\beta:A^* \to A^*\otimes A^*$ is a 1-cocycle of $A^{*c}$ with coefficients in the representation $\Big(A\otimes A;L\otimes \Id+\Id \otimes (L+R)\Big)$, by \eqref{coboundary-condition}, for all $x,y\in A, \xi,\eta\in A^*$, we have
\begin{eqnarray*}
&& \langle \beta(\xi\ast_{A^*} \eta)-(\Id \otimes \huaL_{\eta}+\huaR_{\eta})\beta(\xi)-(L_{\xi}\otimes \Id)\beta(\eta), x\otimes y\rangle\\
&=& \langle\xi\ast_{A^*} \eta, x\circ_A y
\rangle+\langle \xi,x\circ_A (\huaR_{\eta}^* y) \rangle+\langle \xi,x \circ_A (\huaL_{\eta}^* y) \rangle+\langle \eta, (\huaL_{\xi}^* x)\circ_A y  \rangle\\
&=& \langle\xi\ast_{A^*} \eta, x\circ_A y\rangle+\langle L_x^*(\xi) \cdot_{A^*} \eta+\eta \cdot_{A^*} L_x^*(\xi)-L^*(\huaL^*_{\xi}x)\eta,y \rangle\\
&=& \langle \xi\otimes\eta, \alpha(x\circ_A y)+\sigma \alpha(x\circ_A y)-(L_x\otimes \Id)(\alpha(y)+\sigma\alpha(y))-(\Id\otimes R_y)\alpha(x)\rangle\\
&=& (L_{x\circ_A y}\otimes \Id-\Id \otimes (L_{x\circ_A y}+R_{x\circ_A y}))r+(\Id\otimes L_{x\circ_A y}-(L_{x\circ_A y}+R_{x\circ_A y})\otimes \Id)\sigma(r)\\
&&-(L_x\otimes \Id)(L_y\otimes \Id-\Id \otimes (L_y+R_y))r-(L_x\otimes \Id)(\Id\otimes L_y- (L_y+R_y)\otimes \Id)\sigma(r)\\
&&-(\Id\otimes R_y)(L_x\otimes \Id-\Id \otimes (L_y+R_y))r\\
&=& \Big( L_{x\circ_A y}\otimes \Id-\Id \otimes L_{x\circ_A y}-L_x L_y\otimes \Id+L_x\otimes L_y \Big)(r-\sigma r)\\
&=& 0.
\end{eqnarray*}
Thus, the proof is finished.
\end{proof}

For any linear map $\alpha:A\to A\otimes A$, let $J_{\alpha}:A\to A\otimes A\otimes A$ be a linear map given by
\begin{eqnarray}\label{J}
J_{\alpha}(x)=(\Id \otimes \alpha)\alpha(x)-(\alpha\otimes \Id)\alpha(x)-(\sigma \otimes \Id)(\alpha\otimes \Id)\alpha(x),\quad \forall x\in A.
\end{eqnarray}

\begin{lem}
Let A be a vector space and $\alpha:A\to A\otimes A$ a linear map. Then $\alpha^*:A^*\otimes A^*\to A^*$ defines a Zinbiel algebra structure on $A^*$ if and only if $J_{\alpha}=0.$
\end{lem}

\begin{proof}
Actually, for all $x\in A$ and $~\xi,\eta,\theta\in A^*$, we have
\begin{eqnarray*}
&&\langle \xi\cdot_{A^*} (\eta\cdot_{A^*}\theta)-(\xi\cdot_{A^*} \eta+\eta\cdot_{A^*}\xi ) \cdot_{A^*}\theta,x \rangle\\
&=& \langle \xi\otimes \eta \otimes \theta, (\Id \otimes \alpha)\alpha(x)-(\alpha\otimes \Id)\alpha(x)-(\sigma \otimes \Id)(\alpha\otimes \Id)\alpha(x) \rangle\\
&=& \langle \xi\otimes \eta \otimes \theta,J_{\alpha}(x) \rangle,
\end{eqnarray*}
which implies that $(A^*,\alpha^*=\cdot_{A^*})$ is a Zinbiel algebra if and only if $J_{\alpha}=0.$
\end{proof}

\begin{lem}\label{QQQ}
Let $(A,\circ_A)$ be a Zinbiel algebra. Writing $r=\sum_{i} a_i\otimes b_i\in A\otimes A$ with $a_i,b_i\in A$. If $\alpha:A\otimes A\otimes A$ is defined by \eqref{coboundary-condition}, then
we have
\begin{equation}\label{Q}
J_{\alpha}(x)=H(x)\llbracket r,r\rrbracket-\sum_{j}\Big( L_{x\circ_A a_j}\otimes \Id-\Id \otimes L_{x\circ_A a_j}-L_x L_{a_j}\otimes \Id+L_x\otimes L_{a_j} \Big)(r-\sigma r)\otimes b_j,
\end{equation}
where $\llbracket r,r\rrbracket=-r_{13}\circ r_{12}-r_{23}\circ r_{21}+r_{13}\ast r_{21}+r_{12}\ast r_{23}-r_{13}\ast r_{23}$ and $H(x)=L_x\otimes \Id\otimes \Id-\Id \otimes \Id \otimes (L_x+R_x)$, for all $x\in A.$
\end{lem}

\begin{proof}
By \eqref{coboundary-condition} and \eqref{J}, for all $x\in A$, we have
\begin{eqnarray*}
&&J_{\alpha}(x)\\
&=&(\Id \otimes \alpha)\alpha(x)-(\alpha\otimes \Id)\alpha(x)-(\sigma \otimes \Id)(\alpha\otimes \Id)\alpha(x)\\
&=&\sum_{i,j}(x\circ_A a_i)\otimes (b_i\circ_A a_j)\otimes b_j-(x\circ_A a_i)\otimes a_j\otimes (b_i\ast_A b_j)-a_i\otimes (x\ast_A b_i)\circ_A a_j\otimes b_j\\
&&+a_i\otimes a_j\otimes (x\ast_A b_i)\ast_A b_j-(x\circ_A a_i)\circ_A a_j\otimes b_j\otimes b_i+a_j\otimes (x\circ_A a_i)\ast_A b_j\otimes b_i\\
&&+(a_i \circ_A a_j)\otimes b_j\otimes (x\ast_A b_i)-a_j\otimes (a_i\ast_A b_j)\otimes (x\ast_A b_i)-b_j\otimes (x\circ_A a_i)\circ_A a_j\otimes b_i\\
&&+(x\circ_A a_i)\ast_A b_j\otimes a_j\otimes b_i+b_j\otimes (a_i\circ_A a_j)\otimes (x\ast_A b_i)-(a_i\ast_A b_j)\otimes a_j\otimes (x\ast_A b_i)\\
&=& \sum_{i,j}(x\circ_A a_i)\otimes (b_i\circ_A a_j)\otimes b_j-(x\circ_A a_i)\otimes a_j\otimes (b_i\ast_A b_j)-\underline{a_i\otimes x\circ_A (b_i\circ_A a_j) \otimes b_j}\\
&&+a_i\otimes a_j\otimes x\ast_A (b_i\ast_A b_j)-(x\circ_A a_i)\circ_A a_j\otimes b_j\otimes b_i+a_j\otimes (x\circ_A a_i)\circ_A b_j\otimes b_i\\
&&+\underline{a_j\otimes b_j\circ_A(x\circ_A a_i)\otimes b_i}+(a_i \circ_A a_j)\otimes b_j\otimes (x\ast_A b_i)-a_j\otimes (a_i\ast_A b_j)\otimes (x\ast_A b_i)\\
&&-b_j\otimes (x\circ_A a_i)\circ_A a_j\otimes b_i+(x\circ_A a_i)\circ_A b_j\otimes a_j\otimes b_i+b_j \circ_A (x\circ_A a_i)\otimes a_j\otimes b_i\\
&&+b_j\otimes (a_i\circ_A a_j)\otimes (x\ast_A b_i)-(a_i\ast_A b_j)\otimes a_j\otimes (x\ast_A b_i)\\
&=& \sum_{i,j} \Big( (x\circ_A a_i)\otimes (b_i\circ_A a_j)\otimes b_j-(x\circ_A a_i)\otimes a_j\otimes (b_i\ast_A b_j)+a_i\otimes a_j\otimes x\ast_A (b_i\ast_A b_j)\\
&& +(a_i \circ_A a_j)\otimes b_j\otimes (x\ast_A b_i)-a_j\otimes (a_i\ast_A b_j)\otimes (x\ast_A b_i)+b_j \circ_A (x\circ_A a_i)\otimes a_j\otimes b_i\\
&& +b_j\otimes (a_i\circ_A a_j)\otimes (x\ast_A b_i)-(a_i\ast_A b_j)\otimes a_j\otimes (x\ast_A b_i)-x\circ_A (a_j\circ_A a_i)\otimes b_i\otimes b_j\\
&&+x\circ_A (a_j\circ_A b_i)\otimes a_i\otimes b_j+x\circ_A a_i\otimes a_j\circ_A b_i \otimes b_j
-x\circ_A b_i\otimes a_j\circ_A a_i\otimes b_j \Big)\\
&& -\Big(  (x\circ_A a_i)\circ_A a_j\otimes b_j\otimes b_i+a_i\otimes (x\circ_A a_j)\circ_A b_i\otimes b_j-b_i\otimes (x\circ_A a_j)\circ_A a_i\otimes b_j\\
&&+(x\circ_A a_j)\circ_A b_i\otimes a_i\otimes b_j+
+x\circ_A (a_j\circ_A a_i)\otimes b_i\otimes b_j-x\circ_A (a_j\circ_A b_i)\otimes a_i\otimes b_j\\
&&-x\circ_A a_i\otimes a_j\circ_A b_i \otimes b_j+x\circ_A b_i\otimes a_j\circ_A a_i\otimes b_j \Big)\\
&=& (L_x\otimes \Id \otimes \Id)(r_{12}\circ r_{23})-(L_x\otimes \Id \otimes \Id)(r_{13}\ast r_{23})+(\Id\otimes \Id \otimes (L_x+R_x))(r_{13}\ast r_{23})\\
&&+(\Id\otimes \Id \otimes (L_x+R_x))(r_{13}\circ r_{12})-(\Id\otimes \Id \otimes (L_x+R_x))(r_{23}\ast r_{12})+(L_x\otimes \Id \otimes \Id)(r_{21}\circ r_{13})\\
&&+(\Id\otimes \Id \otimes (L_x+R_x))(r_{23}\circ r_{21})-(\Id\otimes \Id \otimes (L_x+R_x))(r_{13}\ast r_{21})-(L_x\otimes \Id \otimes \Id)(r_{13}\circ r_{12})\\
&&+(L_x\otimes \Id \otimes \Id)(r_{13}\circ r_{21})+(L_x\otimes \Id \otimes \Id)(r_{23}\circ r_{12})-(L_x\otimes \Id \otimes \Id)(r_{23}\circ r_{21})\\
&& -\sum_{j}\Big( L_{x\circ_A a_j}\otimes \Id-\Id \otimes L_{x\circ_A a_j}-L_x L_{a_j}\otimes \Id+L_x\otimes L_{a_j} \Big)(r-\sigma(r))\otimes b_j\\
&=& H(x)\llbracket r,r\rrbracket-\sum_{j}\Big( L_{x\circ_A a_j}\otimes \Id-\Id \otimes L_{x\circ_A a_j}-L_x L_{a_j}\otimes \Id+L_x\otimes L_{a_j} \Big)(r-\sigma(r))\otimes b_j.
\end{eqnarray*}
Hence the conclusion holds.
\end{proof}

Based on the above observation, we have the following theorem, which gives a equivalent condition of a coboundary Zinbiel bialgebra.

\begin{thm}\label{equ-coboundary-Zin-bialg}
Let $A$ be a Zinbiel algebra. Writing $r$ as $r=a+\Lambda$ with $a\in \wedge^2 A$ and $ \Lambda\in S^2(A)$. Then the map $\alpha$ defined by \eqref{coboundary-condition} induces a Zinbiel algebra structure on $A^*$ such that $(A,A^*)$ is a Zinbiel bialgebra if and only if the following conditions hold for all $x,y\in A$:
\begin{itemize}
\item[{\rm(i)}] $\Big( L_{x\circ_A y}\otimes \Id-\Id \otimes L_{x\circ_A y}-L_x L_y\otimes \Id+L_x\otimes L_y \Big)(a)=0.$

\item[{\rm(ii)}] $H(x)\llbracket r,r\rrbracket=0,$
\end{itemize}
where $\llbracket r,r\rrbracket$ and $H(x)$ are defined by Lemma \ref{QQQ}.
\end{thm}

The equation $\llbracket r,r\rrbracket=-r_{13}\circ r_{12}-r_{23}\circ r_{21}+r_{13}\ast r_{21}+r_{12}\ast r_{23}-r_{13}\ast r_{23}=0$ is called the {\bf Zinbiel Yang-Baxter equation}, as an analogue of the classical Yang-Baxter Equation.

\section{Quasi-triangular and factorizable Zinbiel bialgebras}\label{sec:factorizable}

In this section, we introduce the notions of quasi-triangular Zinbiel bialgebras and factorizable Zinbiel bialgebras. We show that the double of a Zinbiel bialgebra enjoys a factorizable Zinbiel bialgebra naturally.

\subsection{Quasi-triangular Zinbiel bialgebras and relative Rota-Baxter operators}

For $r\in A\otimes A,$ we define $r_+,r_-:A^*\to A$ by
\begin{eqnarray}
\langle r_+(\xi), \eta\rangle=r(\xi,\eta)=\langle \xi, r_-(\eta) \rangle, \quad \forall \xi,\eta\in A^*.
\end{eqnarray}

Then the Zinbiel algebra structure on $A^*$ defined by \eqref{coboundary-condition} in Theorem \ref{equ-coboundary-Zin-bialg} is given by
\begin{eqnarray}\label{Ar}
\xi \cdot_r \eta=-(L_{r_+(\xi)}^*+R_{r_+(\xi)}^*)\eta+R_{r_{-}(\eta)}^* \xi,\quad \forall \xi,\eta\in A^*.
\end{eqnarray}

Now we introduce the notion of  $(L,L+R)$-invariance of a 2-tensor $r\in A\otimes A$, which is the main ingredient in the definition of a quasi-triangular Zinbiel bialgebra.
\begin{defi}
Let $(A,\circ_A)$ be a Zinbiel algebra and $r\in A\otimes A.$ Then $r$ is  {\bf $(L,L+R)$-invariant} if
\begin{equation}
\Big(L_x \otimes \Id-\Id\otimes (L_x+R_x) \Big)r=0, \quad \forall x \in A.
\end{equation}
\end{defi}

\begin{lem}\label{invariance1}
Let $(A,\circ_A)$ be a Zinbiel algebra and $r\in A\otimes A.$ Then $r$ is $(L,L+R)$-invariant if and only if
$$r_+(L_x^{\ast}\xi)+x \ast_A r_+(\xi)=0, \quad \forall x\in A,\xi \in A^{\ast}. $$
\end{lem}

\begin{proof}
By a direct calculation, for all $x\in A, \xi,\eta \in A^{\ast}$, we have
\begin{eqnarray*}
\langle  r_+(L_x^{\ast}\xi)+x \ast_A r_+(\xi),\eta \rangle
&=& \langle r, L_x^{\ast}\xi\otimes\eta\rangle-\langle r, \xi\otimes (L_x^*+R_x^*)\eta    \rangle\\
&=& - \langle  (L_x \otimes \Id)r,\xi \otimes \eta   \rangle+\langle  (\Id\otimes (L_x+R_x)  )r,\xi \otimes \eta   \rangle\\
&=& - \langle  \Big(L_x \otimes \Id-\Id\otimes (L_x+R_x) \Big)r,  \xi \otimes \eta   \rangle,
\end{eqnarray*}
which implies that $r$ is $(L,L+R)$-invariant if and only if $r_+(L_x^{\ast}\xi)+x \ast_A r_+(\xi)=0$.
\end{proof}

\begin{pro}
Let $(A,\circ_{A})$ be a Zinbiel algebra and $r\in A\otimes
A.$ If the skew-symmetric part $a$ of $r$ is $(L,L+R)$-invariant,
then we have $\Big( L_{x\circ_A y}\otimes \Id-\Id \otimes L_{x\circ_A y}-L_x L_y\otimes \Id+L_x\otimes L_y \Big)(a)=0,$ for all $x,y\in A$.
\end{pro}

\begin{proof}
Since the skew-symmetric part $a$ of $r$ is $(L,L+R)$-invariant, by Lemma \ref{invariance1}, we have
\begin{eqnarray*}
&&\Big( L_{x\circ_A y}\otimes \Id-\Id \otimes L_{x\circ_A y}-L_x L_y\otimes \Id+L_x\otimes L_y \Big)a(\xi,\eta)\\
&=& -a(L_{x\circ_A y}^*\xi, \eta)+a(\xi,L_{x\circ_A y}^*\eta)-a(L_y^* L_x^* \xi, \eta)+a(L_x^*\xi, L_y^*\eta)\\
&=& -\langle a_+(L_{x\circ_A y}^*\xi), \eta\rangle+\langle a_+(\xi),L_{x\circ_A y}^*\eta\rangle-\langle  a_+(L_y^* L_x^* \xi), \eta\rangle +\langle  a_+(L_x^*\xi), L_y^*\eta   \rangle\\
&=& \langle(x\circ_A y)\ast_A a_+(\xi)- (x\circ_A y)\circ_A a_+(\xi)- y\ast_A(x \ast_A a_+(\xi)) + y\circ_A (x \ast_A a_+(\xi) ),\eta\rangle\\
&=& \langle a_+(\xi)\circ_A (x \circ_A y)-(x\ast_A a_+(\xi))\circ_A y,\eta \rangle\\
&=& \langle (a_+(\xi)\ast_A x) \circ_A y-(x\ast_A a_+(\xi))\circ_A y,\eta \rangle\\
&=& 0,
\end{eqnarray*}
for all $x,y\in A$ and $\xi,\eta \in A^*$.
\end{proof}

\begin{defi}
Let $(A,\circ_A)$ be a Zinbiel algebra. If $r\in A\otimes A$ satisfies $\llbracket r,r\rrbracket=0$ and the skew-symmetric part $a$ of $r$ is $(L,L+R)$-invariant, then the Zinbiel bialgebra $(A,A_{r}^{\ast})$ induced by $r$ is called a {\bf quasi-triangular Zinbiel bialgebra}.
Moreover, if $r$ is symmetric, then $(A,A_{r}^{\ast})$ is called a {\bf triangular Zinbiel bialgebra}.
\end{defi}

Denote by $I$ the operator
\begin{eqnarray}\label{I}
I=r_{+}-r_{-}:A^{\ast}\longrightarrow A.
\end{eqnarray}

Note that $I^{\ast}=-I$. Actually, $\frac{1}{2} I$ is the contraction of the skew-symmetric part $a$ of $r$, which means that $\frac{1}{2} \langle I(\xi),\eta \rangle=a(\xi,\eta).  $
If $r$ is symmetric, then $I=0$.

Now we give another characterization of the $(L,L+R)$-invariant condition.

\begin{pro}\label{invariance2}
Let $(A,\circ_{A})$ be a Zinbiel algebra and $r\in A\otimes
A.$ The skew-symmetric part $a$ of $r$ is $(L,L+R)$-invariant if
and only if $I\circ L_x^{\ast}=-(L_x+R_x) \circ I$, or $ L_x\circ I =-I
\circ (L_x^*+R_x^*),$  for all $x\in A$,  where $I$ is
given by \eqref{I}.
\end{pro}

\begin{proof}
By Lemma \ref{invariance1}, the skew-symmetric part $a$ of $r$ is
$(L,L+R)$-invariant if and only if $a_+(L_x^{\ast}\xi)+x \ast_A a_+(\xi)=0.$ Note that  $a_{+}=\frac{1}{2} I$. Thus  the
skew-symmetric part $a$  is $(L,\ad)$-invariant if and only if
$I\circ L_x^{\ast}=-(L_x+R_x) \circ I.$  By $I^{\ast}=-I$,  $I\circ
L_x^{\ast}=-(L_x+R_x) \circ I$ if and only if  $ L_x\circ I =-I \circ (L_x^*+R_x^*)$.
\end{proof}

\begin{lem}
Let $(A,\circ_{A})$ be a Zinbiel algebra and $r\in A\otimes A$. Let $r=a+\Lambda$ with $a\in\wedge^2 A$ and $\Lambda \in S^2 (A)$. If $a$ is $(L,L+R)$-invariant, then
\begin{eqnarray}\label{eq:invariant 2}
\langle \xi, x\circ_A a_+(\eta) \rangle+\langle \eta,x\ast_A a_+(\xi)  \rangle=0, \quad \forall
x \in A, \xi,\eta \in A^{\ast}.
\end{eqnarray}
Furthermore, the Zinbiel algebra multiplication $\cdot_{r}$ on $A^{\ast}$ given by \eqref{Ar} reduces to
\begin{eqnarray}\label{r}
    \xi \cdot_{r} \eta= -(L_{\Lambda_{+}(\xi)}^*+ R_{\Lambda_{+}(\xi)}^*)\eta+R_{\Lambda_{+}(\eta)}^{\ast} \xi,\quad \forall \xi,\eta \in A^{\ast}.
\end{eqnarray}
\end{lem}

\begin{proof}
Since the skew-symmetric part $a$ of $r$  is $(L,L+R)$-invariant, by Proposition \ref{invariance2}, we have
\begin{eqnarray*}
\langle \xi, x\circ_A a_+(\eta) \rangle+\langle \eta,x\ast_A a_+(\xi)  \rangle
&=&-\langle L_x^{\ast}\xi , a_+(\eta)\rangle+\langle \eta,x\ast_A a_+(\xi)  \rangle\\
&=&\langle \eta, a_+(L_x^{\ast}\xi)+x\ast_A a_+(\xi) \rangle=0,
\end{eqnarray*}
which implies that \eqref{eq:invariant 2} holds.

Moreover, by \eqref{eq:invariant 2}, we have
\begin{eqnarray*}
\langle -(L_{a_{+}(\xi)}^*+ R_{a_{+}(\xi)}^*)\eta -R_{a_+(\eta)}^{\ast} \xi , x \rangle
=\langle \eta,a_+(\xi)\ast_A x  \rangle+\langle \xi,x\circ_A  a_+(\eta) \rangle
= 0.
\end{eqnarray*}
Then by \eqref{Ar}, we have
\begin{eqnarray*}
\xi \cdot_{r} \eta
&=& -(L_{r_{+}(\xi)}^{\ast}+R_{r_{+}(\xi)}^{\ast}) \eta+R_{r_{-}(\eta)}^{\ast} \xi\\
&=& -(L_{\Lambda_{+}(\xi)}^{\ast}+R_{\Lambda_{+}(\xi)}^{\ast}) \eta-(L_{a_{+}(\xi)}^{\ast}+R_{a_{+}(\xi)}^{\ast}) \eta
+R_{\Lambda_{+}(\eta)}^{\ast} \xi-R_{a_{+}(\eta)}^{\ast} \xi\\
&=& -(L_{\Lambda_{+}(\xi)}^{\ast}+R_{\Lambda_{+}(\xi)}^{\ast}) \eta+R_{\Lambda_{+}(\eta)}^{\ast} \xi,
\end{eqnarray*}
which implies that \eqref{r} holds.
\end{proof}

\begin{thm}\label{heart}
Let $(A,\circ_A)$ be a Zinbiel algebra and $r\in A\otimes A$. Assume that the skew-symmetric part $a$ of $r $ is $(L,L+R)$-invariant. Then $r$ satisfies $\llbracket r,r\rrbracket=0$ if and only if $(A^{\ast},\cdot_{r})$ is a Zinbiel algebra and
the linear maps $r_{+},r_{-}:(A^{\ast},\cdot_{r}) \longrightarrow (A,\circ_A)$ are both Zinbiel algebra homomorphisms.
\end{thm}

\begin{proof}
By a direct calculation, for all $\xi,\eta,\theta\in A^*,$ we have
\begin{eqnarray*}
\llbracket \Lambda,\Lambda\rrbracket(\xi,\eta,\theta)
=\langle \xi,\Lambda_+(\eta)\circ_A \Lambda_+(\theta)  \rangle+\langle \eta,\Lambda_+(\xi)\circ_A \Lambda_+(\theta) \rangle-\langle \theta,\Lambda_+(\xi)\ast_A \Lambda_+(\eta) \rangle.
\end{eqnarray*}
Moreover, by \eqref{eq:invariant 2}, we have
$$ \llbracket a,a\rrbracket(\xi,\eta,\theta)=\langle \eta,a_+(\xi)\circ_A a_+(\theta)\rangle,  $$
and
$$\llbracket a,\Lambda\rrbracket(\xi,\eta,\theta)+\llbracket \Lambda,a \rrbracket(\xi,\eta,\theta)=0,$$
which implies that
\begin{eqnarray*}
\llbracket r,r\rrbracket(\xi,\eta,\theta)=\llbracket \Lambda+a, \Lambda+a \rrbracket(\xi,\eta,\theta)
=\llbracket \Lambda,\Lambda\rrbracket(\xi,\eta,\theta)+\llbracket a,a\rrbracket(\xi,\eta,\theta).
\end{eqnarray*}
Next we show that the following equation holds
\begin{eqnarray}\label{333}
r_+(\xi \cdot_r \eta)-r_+(\xi)\circ_A r_+(\eta)=-\llbracket r,r\rrbracket (\xi,\cdot,\eta).
\end{eqnarray}
On the one hand, by \eqref{r}, we have
\begin{eqnarray*}\label{111}
&&\langle \theta, r_+(\xi \cdot_r \eta)-r_+(\xi)\circ_A r_+(\eta) \rangle\\
\nonumber&=& \langle \theta, (a_++\Lambda_+)\Big(-(L_{\Lambda_{+}(\xi)}^{\ast}+R_{\Lambda_{+}(\xi)}^{\ast}) \eta+R_{\Lambda_{+}(\eta)}^{\ast} \xi \Big)\rangle-\langle \theta,(a_++\Lambda_+)(\xi)\circ_A (a_++\Lambda_+)(\eta)  \rangle\\
\nonumber&=& \langle \eta, \Lambda_+(\xi)\ast_A \Lambda_+(\theta)\rangle-\langle \eta, \Lambda_+(\xi)\ast_A a_+(\theta)  \rangle-\langle\xi, \Lambda_+(\theta)\circ_A \Lambda_+(\eta) \rangle+\langle \xi,a_+(\theta)\circ_A \Lambda_+(\eta) \rangle\\
\nonumber&&-\langle \theta, a_+(\xi)\circ_A  a_+(\eta)  \rangle-\langle \theta, a_+(\xi)\circ_A \Lambda_+(\eta)   \rangle-\langle \theta, \Lambda_+(\xi)\circ_A  a_+(\eta)   \rangle-\langle \theta,\Lambda_+(\xi)\circ_A  \Lambda_+(\eta) \rangle.
\end{eqnarray*}
On the other hand, by \eqref{eq:invariant 2}, we have
\begin{eqnarray*}\label{222}
\llbracket r,r\rrbracket(\xi,\theta,\eta)&=&\llbracket \Lambda,\Lambda\rrbracket(\xi,\theta,\eta)+\llbracket a,a\rrbracket(\xi,\theta,\eta)\\
\nonumber&=&\langle \xi,\Lambda_+(\theta)\circ_A \Lambda_+(\eta)  \rangle+\langle \theta,\Lambda_+(\xi)\circ_A \Lambda_+(\eta) \rangle-\langle \eta,\Lambda_+(\xi)\ast_A \Lambda_+(\theta) \rangle\\
\nonumber&&\langle \theta,a_+(\xi)\circ_A a_+(\eta)\rangle,
\end{eqnarray*}
Thus by \eqref{eq:invariant 2},  we have
\begin{eqnarray*}
&&\langle \theta, \llbracket r,r\rrbracket (\xi,\cdot,\eta)+r_+(\xi \cdot_r \eta)-r_+(\xi)\circ_A r_+(\eta)  \rangle\\\
&=&-\underline{\langle \eta, \Lambda_+(\xi)\ast_A a_+(\theta)  \rangle}+\langle \xi,a_+(\theta)\circ_A \Lambda_+(\eta) \rangle-\langle \theta, a_+(\xi)\circ_A \Lambda_+(\eta)   \rangle\underline{-\langle \theta, \Lambda_+(\xi)\circ_A  a_+(\eta)   \rangle}\\
&=& \langle \xi,a_+(\theta)\circ_A \Lambda_+(\eta) \rangle-\langle \theta, a_+(\xi)\circ_A \Lambda_+(\eta)   \rangle\\
&=& \langle \xi, \Lambda_+(\eta) \ast_A a_+(\theta)\rangle-\langle \xi, \Lambda_+(\eta) \circ_A a_+(\theta)\rangle-\langle \theta, a_+(\xi)\circ_A \Lambda_+(\eta)   \rangle\\
&=& -\langle \theta, \Lambda_+(\eta) \circ_A a_+(\xi)\rangle+\langle \theta, \Lambda_+(\eta) \ast_A a_+(\xi)\rangle-\langle \theta, a_+(\xi)\circ_A \Lambda_+(\eta)   \rangle\\
&=& 0,
\end{eqnarray*}
which implies that  \eqref{333} holds. Thus, if $r$ satisfies $\llbracket r,r\rrbracket=0$,
then $r_+$ is a Zinbiel algebra homomorphism. Similarly, we can also prove that  $r_{-}$ is  a Zinbiel algebra homomorphism.

Conversely, if $r_{+}$ is a Zinbiel algebra homomorphism, by \eqref{333}, we have $ \llbracket r,r\rrbracket =0$.
\end{proof}

Next we introduce the notion of relative Rota-Baxter operators on Zinbiel algebras. We need to define actions of Zinbiel algebras, which are similar to actions of Lie algebras.

\begin{defi}
Let $(A,\circ_{A})$ and $(B,\circ_{B})$ be two Zinbiel algebras. An {\bf action} of $(A,\circ_{A})$ on $(B,\circ_{B})$ consists a pair $(\rho,\mu)$, where $(B;\rho,\mu)$ is a representation of a Zinbiel algebra $(A,\circ_{A})$ and two linear maps $\rho,\mu:A \to \End(B)$ satisfy the following equations additionally:
\begin{eqnarray}
\label{action1} u\circ_{B} (\mu(x)v)&=&\mu(x)(u\circ_B v+v\circ_B u);\\
\label{action2} u\circ_{B} (\rho(x)v)&=&(\mu(x)u)\circ_B v+(\rho(x)u)\circ_B v;\\
\label{action3} \rho(x)(u\circ_B v)&=&(\mu(x)u)\circ_B v+(\rho(x)u)\circ_B v,
\end{eqnarray}
for all $x\in A,u,v\in B$.
\end{defi}

\begin{rmk}
Similar to representations of Zinbiel algebras, an action $(\rho,\mu)$ of $(A,\circ_{A})$ on $(B,\circ_{B})$ has an equivalent characterization in terms of the Zinbiel algebra structure on the direct sum $A\oplus B$ of vector spaces. Here the multiplication is given by
$$ (x+u)\circ_{(\rho,\mu)} (y+v)=x\circ_A y+\rho(x)v+\mu(y)u+u\circ_B v, \quad \forall x,y\in A,u,v\in B.$$
\end{rmk}

\begin{defi}
Let $(A,\circ_{A})$ and $(B,\circ_{B})$ be two Zinbiel algebras and $(\rho,\mu)$ be an action of $(A,\circ_{A})$ on $(B,\circ_{B})$. Then a linear map $T:B\to A$ is called a {\bf relative Rota-Baxter operator of weight $\lambda$} with respect to the action $(\rho,\mu)$ of $A$ on $B$, if $T$ satisfies the following equation:
\begin{equation}\label{rRBO}
(Tu)\circ_{A} (Tv)=T(\rho(Tu)v+\mu(Tv)u+\lambda u\circ_B v), \quad \forall u,v\in B.
\end{equation}
\end{defi}

\begin{rmk}
If the multiplication on $B$ is trivial, then we can define a relative Rota-Baxter operator of weight zero, which is also called an $\huaO$-operator. See \cite{BGN,Ku} for more details.
\end{rmk}

By $I=r_+ - r_-$, the Zinbiel algebra multiplication $\cdot_{r}$ on $A^*$ given by \eqref{Ar} reduces to
\begin{equation}\label{rRBO-Ar}
\xi \cdot_{r} \eta=-(L_{r_+ \xi}^*+R_{r_+ \xi}^*)\eta+R_{r_+ \eta}^* \xi-R_{I\eta}^* \xi, \quad \forall \xi,\eta\in A^*.
\end{equation}

Now we define a new multiplication $\cdot_+$ on $A^*$ as follows:
\begin{equation}
\xi \cdot_+ \eta=R_{I\eta}^* \xi, \quad \forall \xi,\eta\in A^*.
\end{equation}

\begin{pro}\label{dot+}
$(A^*,\cdot_+)$ is a Zinbiel algebra.
\end{pro}

\begin{proof}
By \eqref{left-com} and Proposition \ref{invariance2}, for all $x\in A,\xi,\eta,\theta\in A^*,$  we have
\begin{eqnarray*}
&&\langle \xi \cdot_+(\eta \cdot_+ \theta)-(\xi \cdot_+ \eta)\cdot_+ \theta-(\eta \cdot_+ \xi)\cdot_+ \theta, x\rangle\\
&=& \langle R^*(I(R^*_{I\theta} \eta)) \xi-R^*_{I\theta}R^*_{I\eta} \xi-R^*_{I\theta}R^*_{I\xi} \eta,x\rangle\\
&=& \langle \xi, -x\circ_{A} (I\eta \circ_{A} I\theta )-(x\circ_{A} I \theta)\circ_A (I\eta)+(x\circ_{A} I \theta)\circ_A (I\eta)+(I\eta)\circ_A (x\circ_{A} I\theta)\rangle\\
&=& \langle \xi, -x\circ_{A} (I\eta \circ_{A} I\theta )+(I\eta)\circ_A (x\circ_{A} I\theta)\rangle\\
&=& 0,
\end{eqnarray*}
which implies that $\xi \cdot_+(\eta \cdot_+ \theta)=(\xi \cdot_+ \eta)\cdot_+ \theta+(\eta \cdot_+ \xi)\cdot_+ \theta$. Thus, $(A^*,\cdot_+)$ is a Zinbiel algebra.
\end{proof}

\begin{lem}
$(-L^*-R^*,R^*)$ is an action of the Zinbiel algebra $(A,\circ_A)$ on the Zinbiel algebra $(A^*,\cdot_+)$ given in Proposition \ref{dot+}.
\end{lem}

\begin{proof}
We need only to prove that  $(-L^*-R^*,R^*)$ satisfies the equations \eqref{action1}-\eqref{action3}. By \eqref{left-com} and Proposition \ref{invariance2}, for all $x,y \in A,\xi,\eta\in A^*,$  we have
\begin{eqnarray*}
&&\langle \xi \cdot_+ (R^*_x \eta)-R^*_x (\xi\cdot_+ \eta)-R^*_x (\eta\cdot_+ \xi),y \rangle\\
&=& \langle R^*(I(R^*_x \eta))\xi-R^*_x R^*_{I\eta} \xi-R^*_x R^*_{I\xi} \eta,y \rangle\\
&=& \langle \xi, -y\circ_{A} ((I\eta) \circ_A x)-(y \circ_A x)\circ (I\eta)+(y \circ_A x)\circ (I\eta)+(I\eta)\circ_A (y\circ_A x)\rangle\\
&=& \langle \xi, -y\circ_{A} ((I\eta) \circ_A x)+(I\eta)\circ_A (y\circ_A x)\rangle\\
&=& 0,
\end{eqnarray*}
which implies that \eqref{action1} is satisfied. Similarly, equations \eqref{action2} and  \eqref{action3} hold by direct calculation. Therefore,
$(-L^*-R^*,R^*)$ is an action of the Zinbiel algebra $(A,\circ_A)$ on  $(A^*,\cdot_+)$.
\end{proof}

Then we show that quasi-triangular Zinbiel bialgebras can give rise to relative Rota-Baxter operators of weight $-1$ with respect to the coregular representation $(-L^*-R^*,R^*)$.

\begin{thm}\label{quasi-rRBO}
Let $(A,A_r^*)$ be a quasi-triangular Zinbiel bialgebra induced by $r\in A\otimes A$. Then $r_+:(A^*,\cdot_+) \to (A,\circ_A)$ is a relative Rota-Baxter operators of weight $-1$ with
respect to the action $(-L^*-R^*,R^*)$.
\end{thm}

\begin{proof}
By Theorem \ref{heart} and \eqref{rRBO-Ar}, for all $\xi,\eta\in A^*,$ we have
\begin{eqnarray*}
(r_+\xi)\circ_A (r_+\eta)&=& r_+(\xi\cdot_{r} \eta)\\
&=& r_+(-(L_{r_+ \xi}^*+R_{r_+ \xi}^*)\eta+R_{r_+ \eta}^* \xi-R_{I\eta}^* \xi)\\
&=& r_+(-(L_{r_+ \xi}^*+R_{r_+ \xi}^*)\eta+R_{r_+ \eta}^* \xi-\xi\cdot_+ \eta),
\end{eqnarray*}
which implies that $r_+:(A^*,\cdot_+) \to (A,\circ_A)$ is a relative Rota-Baxter operators of weight $-1$ with respect to the action $(-L^*-R^*,R^*)$.
\end{proof}

\begin{rmk}
In fact, if we define a new multiplication on $A^*$ as follows:
$$ \xi \cdot_- \eta=L_{I\xi}^*\eta + R_{I\xi}^*\eta, \quad \forall \xi,\eta\in A^*,$$
then $(A^*,\cdot_-)$ is also a Zinbiel algebra. We can also prove that $r_-:(A^*,\cdot_-) \to (A,\circ_A)$ is a relative Rota-Baxter operators of weight $-1$ with respect to the action $(-L^*-R^*,R^*)$.
\end{rmk}

Taking a symmetric $r\in A\otimes A$ into Theorem \ref{quasi-rRBO}, we can obtain that a triangular Zinbiel bialgebra can give rise to a relative Rota-Baxter operators of weight $0$ (also called $\huaO$-operator), with respect to the coregular representation $(-L^*-R^*,R^*)$.

\begin{cor}
Let $(A,A_r^*)$ be a triangular Zinbiel bialgebra induced by a symmetric $r\in A\otimes A$. Then $r_+,r_-:A^* \to A$ are relative Rota-Baxter operators of weight $0$ with
respect to the coregular representation $(-L^*-R^*,R^*)$.
\end{cor}

\subsection{Factorizable Zinbiel bialgebras}

\begin{defi}
A quasi-triangular Zinbiel bialgebra $(A,A_{r}^{\ast})$ is called {\bf factorizable} if the skew-symmetric part $a$ of $r$ is non-degenerate, which means that the linear map $I:A^{\ast} \longrightarrow A$ defined in \eqref{I} is a linear isomorphism of vector spaces.
\end{defi}

Consider the map
$$
A^*\stackrel{r_+\oplus r_-}{\longrightarrow}A\oplus A\stackrel{(x,y)\longmapsto x-y}{\longrightarrow}A.
$$
The following result justifies the terminology of a factorizable Zinbiel bialgebra.
\begin{pro}
  Let  $(A,A_{r}^{\ast})$ be a factorizable Zinbiel bialgebra. Then $\Img(r_+\oplus r_-)$ is a Zinbiel subalgebra of the direct sum Zinbiel algebra $A \oplus A$, which is isomorphic to the Zinbiel algebra $(A^*,\cdot_r).$ Moreover, any $x\in A$ has a unique decomposition
  \begin{equation}
    x=x_+-x_-,
  \end{equation}
  where $(x_+,x_-)\in \Img(r_+\oplus r_-)$.
\end{pro}
\begin{proof}
By Theorem \ref{heart}, both $r_+$ and $r_-$ are Zinbiel algebra homomorphisms. Therefore, $\Img(r_+\oplus r_-)$ is a Zinbiel subalgebra of the direct sum Zinbiel algebra $A \oplus A$. Since $I=r_+  -r_-:A^*\to A$ is nondegenerate, it follows that the Zinbiel algebra $\Img(r_+\oplus r_-)$ is   isomorphic to the Zinbiel algebra $(A^*,\cdot_r).$

  Since $I:A^*\to A$ is nondegenerate, we have
  $$
 r_+ I^{-1}(x)-r_-I^{-1}(x)= (r_+  -r_-)I^{-1}(x)=x,
  $$
  which implies that $x=x_+-x_-,$ where $x_+=r_+ I^{-1}(x)$ and $x_-=r_-I^{-1}(x)$. The uniqueness also follows from the fact that $I:A^*\to A$ is nondegenerate.
\end{proof}

\begin{rmk}
Actually, $-r_-I^{-1}:A \to A$ is an idempotent Rota-Baxter operator of weight $-1$, which can be seen from the proof of the following Theorem \ref{Factorizable Zinbiel algebra}. The factorization given above coincides with the factorization given by the idempotent Rota-Baxter
operator $-r_-I^{-1}$ of weight $-1$.  See \cite{Guo} for more details for the latter.
\end{rmk}

\subsection{Zinbiel double}

Let $(A,A^*)$ be a Zinbiel bialgebra. Define a multiplication $\cdot_{\frkd}$ on $\frkd=A\oplus A^*$ by
\begin{equation}
(x,\xi)\cdot_{\frkd}(y,\eta)=\Big(  x\circ_A y-(\huaL^*_{\xi}+\huaR^*_{\xi})y+\huaR^*_{\eta}(x),\xi\cdot_{A^*} \eta-(L_x^*+R_x^*)\eta+R_y^*(\xi) \Big), \quad \forall x,y\in A,\xi,\eta\in A^*.
\end{equation}
Actually, $(\frkd,\cdot_{\frkd})$ is a special matched pair $(A,A^*;-L^*-R^*,R^*,-\huaL^*-\huaR^*,\huaR^*)$, or simply denoted by $\frkd=A \bowtie_{-L^*-R^*,R^*}^{-\huaL^*-\huaR^*,\huaR^*} A^*$. By Proposition \ref{mp-equ}, $(\frkd,\cdot_{\frkd})$ is a Zinbiel algebra, which is called the {\bf Zinbiel double} of the  Zinbiel bialgebra $(A,A^*).$

\begin{thm}\label{Z-double-fac}
Let $(A,A^*)$ be a Zinbiel bialgebra. Suppose that $\{ e_1,e_2,\dots,e_n\}$ is a basis of $A$ and $\{ e_1^*,e_2^*,\dots,e_n^*\}$ is a basis of $A^*$. Then $r=\sum_{i} e_i\otimes e_i^* \in A\otimes A^*  \subset \frkd \otimes \frkd$ induces a Zinbiel algebra structure on $\frkd^*=A^*\oplus A$ such that $(\frkd,\frkd_r^*)$ is a quasi-triangular Zinbiel bialgebra, where the Zinbiel algebra structure $(\frkd_r^*,\cdot_{\frkd_r^*})$ is given by
\begin{equation}\label{frkd-r}
(\xi,x)\cdot_{\frkd_r^*} (\eta,y)=(\xi \cdot_{A^*} \eta, x \circ_A y),\quad \forall x,y\in A,\xi,\eta\in A^*.
\end{equation}
Moreover, $(\frkd,\frkd_r^*)$ is also a factorizable Zinbiel bialgebra.
\end{thm}

\begin{proof}
We first prove that the skew-symmetric part $a=\frac{1}{2} \sum_{i} (e_i\otimes e_i^{\ast}-e_i^{\ast}\otimes e_i) $ of $r$ is $(\tilde{L},\tilde{L}+\tilde{R})$-invariant, where $\tilde{L},\tilde{R}$ are the left and right multiplication operators of the Zinbiel algebra $(\frkd,\cdot_\frkd)$ respectively.  For all  $(\xi, x)\in \frkd^{\ast}$, we have $a^{\sharp}(\xi, x)=\frac{1}{2} (-x,\xi)\in \frkd.$ Furthermore, by a direct calculation,  we have
\begin{eqnarray*}
(x,\xi)\ast_{\frkd} a^{\sharp}(\eta,y)
 &=& \frac{1}{2} \Big( -x\ast_A y-\huaL_{\eta}^*x+\huaL_{\xi}^*y,\xi\ast_{A^*} \eta-L_x^*\eta+L_y^* \xi \Big),\\
\tilde{L}^{\ast}_{(x,\xi)} (\eta,y)&=&\Big(-\xi \ast_{A^*} \eta+L_x^*\eta-L_y^* \xi,-x\ast_A y-\huaL_{\eta}^*x+\huaL_{\xi}^*y \Big).
\end{eqnarray*}
Thus we have
$$a^{\sharp} \Big(\tilde{L}^{\ast}_{(x,\xi)}(\eta,y)\Big)+(x,\xi)\ast_{\frkd} a^{\sharp}(\eta,y)=0.$$
By Lemma \ref{invariance1},  the skew-symmetric part $a$ of $r$ is  $(\tilde{L},\tilde{L}+\tilde{R})$-invariant. We also have
\begin{eqnarray*}
\llbracket r,r\rrbracket &=& \sum_{i,j}(-e_i\circ_A e_j\otimes e_j^* \otimes e_i^*-e_j^*\otimes e_i\circ_A e_j\otimes e_i^*+e_i \ast_{\frkd} e_j^*\otimes e_j \otimes e_i^*+e_i \otimes e_i^* \ast_{\frkd} e_j \otimes e_j^*\\
&&-e_i \otimes e_j\otimes e_i^* \ast_{\frkd} e_j^*)\\
&=& \sum_{i,j}(-e_i\circ_A e_j\otimes e_j^* \otimes e_i^*-e_j^*\otimes e_i\circ_A e_j\otimes e_i^*-\huaL^*_{e_j^*} e_i\otimes e_j\otimes e_i^*-L^*_{e_i} e_j^*\otimes e_j\otimes e_i^*\\
&&-e_i\otimes L^*_{e_j} e_i^*\otimes e_j^*-e_i\otimes \huaL^*_{e_i^*} e_j\otimes e_j^*-e_i \otimes e_j\otimes e_i^* \ast_{\frkd} e_j^*).
\end{eqnarray*}
Note that
\begin{eqnarray*}
&&\sum_{i,j}(e_i\circ_A e_j\otimes e_j^* \otimes e_i^*+e_i\otimes L^*_{e_j} e_i^*\otimes e_j^*)=0;\\
&&\sum_{i,j}(e_j^*\otimes e_i\circ_A e_j\otimes e_i^*+L^*_{e_i} e_j^*\otimes e_j\otimes e_i^*)=0;\\
&&\sum_{i,j}(\huaL^*_{e_j^*} e_i\otimes e_j\otimes e_i^*-e_i\otimes \huaL^*_{e_i^*} e_j\otimes e_j^*-e_i \otimes e_j\otimes e_i^* \ast_{\frkd} e_j^*)=0.\\
\end{eqnarray*}
Thus $\llbracket r,r\rrbracket =0.$ Hence $(\frkd,\frkd^{\ast}_r)$ is a quasi-triangular Zinbiel bialgebra.

Moreover, note that $r_{+},r_{-}:\frkd^{\ast} \longrightarrow
\frkd$ are given   respectively   by
$$ r_{+}(\xi,x)=(0,\xi),\quad r_{-}(\xi,x)=(x,0),\quad \forall x\in A, \xi\in A^{\ast}.$$
This implies that $I(\xi,x)=(-x,\xi)$, which means that the linear map $I:\frkd^{\ast} \longrightarrow \frkd$ is a linear isomorphism of vector spaces. Therefore, $(\frkd,\frkd^{\ast}_r)$ is a factorizable Zinbiel bialgebra.
\end{proof}

\emptycomment{
Actually, for all $(z,\theta)\in \frkd$, we have
\begin{eqnarray*}
&&\langle (\xi,x)\cdot_{\frkd_r^*} (\eta,y),(z,\theta)  \rangle\\
&=& \langle \alpha(z,\theta),(\xi,x)\otimes (\eta,y) \rangle\\
&=& \langle (e_i,0)\otimes (z,\theta)\ast_{\frkd} (0,e_i^*)-(z,\theta)\cdot_{\frkd}(e_i,0)\otimes(0,e_i^*),(\xi,x)\otimes (\eta,y) \rangle\\
&=& \langle \xi,e_i \rangle \Big( \langle e_i^* \cdot_{A^*}\eta,z \rangle+\langle \theta\ast_{A^*} e_i^*,y \rangle+\langle e_i^*,z\circ_{A} y \rangle \Big)\\
&&-\Big( \langle\xi,z\circ_{A} e_i\rangle+\langle \theta\ast_{A^*} \xi,e_i \rangle-\langle\theta,x\circ_{A} e_i\rangle \Big)\langle e_i^*, y\rangle\\
&=& \langle \xi \cdot_{A^*} \eta,z \rangle+\langle \theta,x\circ_{A} y\rangle\\
&=& \langle (\xi \cdot_{A^*} \eta, x \circ_A y) ,(z,\theta)\rangle,
\end{eqnarray*}
which implies that $(\xi,x)\cdot_{\frkd_r^*} (\eta,y)=(\xi \cdot_{A^*} \eta, x \circ_A y)$ is the Zinbiel algebra structure on $\frkd_r^*$.
}

\section{Rota-Baxter characterization of factorizable Zinbiel bialgebras}\label{sec:quadratic-RB}

In this section, first we recall the one-to-one correspondence between quadratic Zinbiel algebras and commutative algebras with Connes cocycles. Then we add Rota-Baxter operators on these algebraic structures, and show that there is still a one-to-one correspondence between quadratic Rota-Baxter Zinbiel algebras and Rota-Baxter commutative associative algebras with Connes cocycles. Finally we show that there is a one-to-one correspondence between factorizable Zinbiel bialgebras and  quadratic Rota-Baxter Zinbiel algebras, as well as Rota-Baxter commutative associative algebras with Connes cocycles.

\subsection{Quadractic Rota-Baxter Zinbiel algebras}

The notion of quadractic Zinbiel algebras is given by definition \ref{quadratic-Zinbiel-alg}.

Recall that a skew-symmetric bilinear form $\omega\in\wedge^2 A^*$ is called a {\bf Connes cocycle} on a commutative associative algebra $(A,\ast_A)$ if
$$
\omega(x\ast_A y,z)+\omega(y\ast_A z,x)+\omega(z\ast_A x,y)=0,\quad  \forall x,y,z \in A.
$$

There is a one-to-one correspondence between quadratic Zinbiel algebras and commutative associative  algebras with nondegenerate Connes cocycles.

\begin{thm}\label{lem:Zin-Connes}{\rm(\cite{LB})}
Let $(A, \circ_A,\omega)$ be a quadratic Zinbiel algebra. Then $(A,\ast_A,\omega)$ is a commutative associative algebras with a nondegenerate Connes cocycle $\omega$, where the commutative associative multiplication $\ast_A$ is given by \eqref{Zin-com}.

Conversely, let $(A,\ast_A,\omega)$ be a commutative associative algebras with a nondegenerate Connes cocycle $\omega$. Then there exists a compatible Zinbiel algebra structure $\circ_A$ on $A$ given by
    \begin{equation}\label{eq:Zin-Connes}
        \omega(x\circ_A y,z)=\omega(y,x \ast_A z),\quad\forall x,y,z\in A,
    \end{equation}
whose sub-adjacent commutative associative algebra is exactly $(A,\ast_A)$. Furthermore, $(A,\circ_A,\omega)$ is a quadratic Zinbiel algebra.
\end{thm}

Next we recall Rota-Baxter operators on Zinbiel algebras and commutative associative algebras.
 A linear map $P:A \to A$ is called a {\bf Rota-Baxter operator of weight $\lambda$} on a Zinbiel algebra  $(A,\circ_A)$ if
$$ P(x)\circ_A P(y)=P(P(x)\circ_A y+x \circ_A P(y)+\lambda x \circ_A y),\quad  \forall x,y\in A. $$
Note that a Rota-Baxter operator of weight $\lambda$ is actually a relative Rota-Baxter operator of the same weight with respect to the regular representation. A {\bf Rota-Baxter Zinbiel algebra} $(A,\circ_A,P)$ of weight $\lambda$ is a Zinbiel algebra equipped with a Rota-Baxter operator of weight $\lambda$.

Let $(A,\circ_A,P)$ be a Rota-Baxter Zinbiel algebra of weight $\lambda$. Then there is a new Zinbiel multiplication $\cdot_{P}$ on $A$ defined by
$$ x \cdot_{P} y= P(x)\circ_A y+x \circ_A P(y)+\lambda x\circ_A y. $$
The Zinbiel algebra $(A,\cdot_{P})$ is called the {\bf descendent Zinbiel algebra} and denoted by $A_{P}$. Furthermore, $P$ is a Zinbiel algebra homomorphism from the Zinbiel algebra $A_{P}$ to $(A,\circ_A)$.

Recall that a {\bf Rota-Baxter operator of weight $\lambda$} on a commutative associative algebra $(A,\ast_A)$ is a linear map $P:A \to A$ satisfying
\begin{equation} \label{eq:Com-RB}
    P(x)\ast_A P(y)=P(P(x)\ast_A y+x \ast_A P(y)+\lambda x \ast_A y),\quad  \forall x,y\in A.
\end{equation}

We add Rota-Baxter operators on quadratic Zinbiel algebras and commutative associative algebras with nondegenerate Connes cocycles, and introduce the notions of quadratic Rota-Baxter Zinbiel algebras and Rota-Baxter commutative associative algebras with nondegenerate Connes cocycles.
\begin{defi}\label{quadratic RB Zinbiel algebra}
    Let $(A,\circ_A,P)$ be a Rota-Baxter Zinbiel algebra of weight $\lambda$ and $(A,\circ_A,\omega)$  a quadratic Zinbiel algebra. Then the quadruple $(A,\circ_A,P,\omega)$ is called a {\bf quadratic Rota-Baxter Zinbiel algebra of weight $\lambda$} if the following compatibility condition holds:
    \begin{eqnarray}\label{compatibility condition}
        \omega(Px,y)+   \omega(x,Py)+\lambda \omega(x,y)=0, \quad \forall x,y \in A.
    \end{eqnarray}
\end{defi}
\begin{defi}
    Let $(A,\ast_A,P)$ be a Rota-Baxter commutative associative algebra of weight $\lambda$ and $(A,\ast_A,\omega)$ be a commutative associative algebras with a nondegenerate Connes cocycle. Then $(A,\ast_A,P,\omega)$ is  called a {\bf Rota-Baxter commutative associative algebras with a nondegenerate Connes cocycle of weight $\lambda$} if $P$ and $\omega$ satisfy the compatibility condition \eqref{compatibility condition}.
\end{defi}

We now show that the relation between quadratic Zinbiel algebras and commutative associative algebras with nondegenerate Connes cocycles given in Theorem \ref{lem:Zin-Connes} also holds for quadratic Rota-Baxter Zinbiel algebras and Rota-Baxter commutative associative algebras with nondegenerate Connes cocycles of the same weight.

\begin{thm}\label{thm:Zin-Connes-RB}
Let $(A,\circ_A,P,\omega)$ be a quadratic Rota-Baxter Zinbiel algebra of weight $\lambda$. Then  $(A,\ast_A,P,\omega)$ is a Rota-Baxter commutative associative algebras with a nondegenerate Connes cocycle of of weight $\lambda$, where $\ast_A$ is given by \eqref{Zin-com}.

Conversely, let $(A,\ast_A,P,\omega)$ be a Rota-Baxter commutative associative algebras with a nondegenerate Connes cocycle of of weight $\lambda$. Then $(A,\circ_A,P,\omega)$ is a quadratic Rota-Baxter Zinbiel algebra of weight $\lambda$, where the Zinbiel algebra multiplication $\circ_A$ is given by \eqref{eq:Zin-Connes}.
\end{thm}

\begin{proof}
It is straightforward to deduce that if $P:A \rightarrow A$ is a  Rota-Baxter operator of weight $\lambda$ on a Zinbiel algebra $(A,\circ_A)$, then $P$  is a  Rota-Baxter operator of weight $\lambda$ on the sub-adjacent commutative associative algebra $(A^c,\ast_A)$. Since  $(A,\circ_A,\omega)$ is a quadratic Zinbiel algebra, by Theorem \ref{lem:Zin-Connes},   $(A,\ast_A,\omega)$ is a commutative associative algebra with a nondegenerate Connes cocycle $\omega$. Therefore, if $(A,\circ_A,P,\omega)$ is a quadratic Rota-Baxter Zinbiel algebra of weight $\lambda$, then $(A,\ast_A,P,\omega)$ is a Rota-Baxter commutative associative algebra with a nondegenerate Connes cocycle of weight $\lambda$.

Conversely, let $(A,\ast_A,P,\omega)$ be a Rota-Baxter commutative associative algebra with a nondegenerate Connes cocycle of weight $\lambda$. First by Theorem \ref{lem:Zin-Connes}, $(A,\circ_A,\omega)$ is a quadratic Zinbiel algebra.  In the following, we show that $P$ is a Rota-Baxter operator of weight $\lambda$ on the Zinbiel algebra $(A,\circ_A)$. By \eqref{eq:Zin-Connes} and \eqref{compatibility condition}, for all $x,y,z\in A,$ we have
    \begin{eqnarray*}
    &&\omega(P(x)\circ_A P(y)-P(P(x)\circ_A y+x \circ_A P(y)+\lambda x \circ_A y),z)\\
    &=& \omega(P(x)\circ_A P(y),z)-\omega(P(P(x)\circ_A y),z)-\omega(B(x\circ_A B(y)),z)-\lambda \omega(B(x \circ_A  y),z)\\
    &=& \omega(P(y),P(x)\ast_A z)+\omega(P(x)\circ_A y,P(z))+\lambda\omega(P(x)\circ_A y,z)+\omega(x\circ_A P(y),P(z))\\
    &&+\lambda\omega(x\circ_A P(y),z)+\lambda\omega(x\circ_A y,P(z))+\lambda^2 \omega(x\circ_A y,z)\\
    &=& -\omega(y,P(P(x)\ast_A z))-\lambda\omega(y,P(x)\ast_A z)+\omega(y,P(x)\ast_A P(z))
    +\lambda\omega(y,P(x)\ast_A z)\\
    &&+\omega(P(y),x\ast_A P(z))+\lambda\omega(P(y),x\ast_A z)+\lambda\omega(y,x\ast_A P(z))+\lambda^2 \omega(y,x\ast_A z)\\
    &=& -\omega(y,P(P(x)\ast_A z))+\omega(y,P(x)\ast_A P(z))-\omega(y,P(x\ast_A P(z)))-\lambda\omega(y,x\ast_A P(z))\\
    && -\lambda\omega(y,P(x\ast_A z))-\lambda^2 \omega(y,x\ast_A z)+\lambda\omega(y,x\ast_A P(z))+\lambda^2 \omega(y,x\ast_A z)\\
    &=& -\omega(y,P(P(x)\ast_A z))+\omega(y,P(x)\ast_A P(z))-\omega(y,P(x\ast_A P(z)))-\lambda\omega(y,P(x\ast_A z))\\
    &=&0.
    \end{eqnarray*}
The last equality follows from the fact that $P$ is a Rota-Baxter operator of weight $\lambda$ on the commutative associative algebra algebra $(A,\ast_A)$. Furthermore, by the nondegeneracy of $\omega$, we have
    $$ P(x)\circ_A P(y)-P(P(x)\circ_A y+x \circ_A P(y)+\lambda x \circ_A y)=0,\quad \forall x,y\in A,  $$
which implies that $P$ is a Rota-Baxter operator of weight $\lambda$ on the Zinbiel algebra $(A,\circ_A)$. Therefore, $(A,\circ_A,P,\omega)$ is a quadratic Rota-Baxter Zinbiel algebra of weight $\lambda$.
\end{proof}

Let $(A,\ast_A)$ be a commutative associative algebra. Recall that $(A,\ast_A,\frkB)$ is a {\bf quadratic commutative associative algebra} if $\frkB$ is a nondegenerate symmetric bilinear form on $A$ such that $\frkB(x\ast_A y,z)=\frkB(x,y \ast_A z)$ for all $x,y,z\in A$.

\begin{ex}{\rm
Let $(A,\ast_A,\frkB,\omega)$ be a quadratic commutative associative algebra with a nondegenerate Connes cocycle $\omega$. Define $D:A\longrightarrow A$ by
$$ \omega(x,y)= \frkB(Dx,y),\quad \forall x,y\in A. $$
It is straightforward to verify that $D$ is a skew-symmetric invertible derivation on the quadratic commutative associative algebra $(A,\ast_A,\frkB)$. Denote the inverse of $D$ by $P$. Then $P$ is a Rota-Baxter operator of weight $0$ on $A$. Moreover, $(A,\ast_A,P,\omega)$ is a Rota-Baxter commutative associative algebra with a nondegenerate Connes cocycle of weight $0$.

Since $D$ is a skew-symmetric invertible derivation on $A$, there exists a Manin decomposition $A=A_+ \oplus A_-$, where $A_+$ and $A_-$ denote the sum of the weight spaces of positive and negative eigenvalues of $D$ {\rm(\cite{AGMM})}. For any $ \lambda \in \mathbb{K}, $ define linear maps $P_+,P_-:A \longrightarrow A$ by
$$ P_+ (x+\xi)=-\lambda x, \quad P_- (x+\xi)=-\lambda \xi, \quad \forall x\in A_+,\xi\in A_-.  $$
Then $P_+$ and $P_-$ are Rota-Baxter operators of weight $\lambda$ on $A$. It is straightforward to check that $(A,\ast_A,P_+,\omega)$ and $(A,\ast_A,P_-,\omega)$ are Rota-Baxter commutative associative algebras with a nondegenerate Connes cocycle of weight $\lambda$.
}
\end{ex}
\subsection{Factorizable Zinbiel bialgebras and quadractic Rota-Baxter Zinbiel algebras}

The following theorem shows that a factorizable Zinbiel bialgebra gives a quadratic Rota-Baxter Zinbiel algebra.
\begin{thm}\label{Factorizable Zinbiel algebra}
Let $(A,A_{r}^{\ast})$ be a factorizable Zinbiel bialgebra with $I=r_{+}-r_{-}$. Then $(A,P,\omega_{I})$ is a quadratic Rota-Baxter Zinbiel algebra of weight $\lambda$, where the linear map $P:A \to A$ and $\omega_{I}\in \wedge^{2} A^{\ast}$ are defined respectively by
\begin{eqnarray}
\label{P}P&=&\lambda r_{-}\circ I^{-1},\\
\label{SI}\omega_I(x,y)&=&\langle I^{-1}x,y \rangle, \quad \forall x,y \in A.
\end{eqnarray}
\end{thm}

\begin{proof}
Since $r_{+},r_{-}:(A^{\ast},\cdot_{r}) \longrightarrow (A,\circ_A)$ are both Zinbiel algebra homomorphisms, for all $x,y\in A$, we have
\begin{eqnarray}\label{RB operator in fac}
I(I^{-1}x \cdot_{r} I^{-1}y )&=&(r_{+}-r_{-})(I^{-1}x \cdot_{r} I^{-1}y )\\
\nonumber&=&((I+r_{-})I^{-1}x) \circ_A ((I+r_{-})I^{-1}y)- (r_{-}I^{-1}x)\circ_A (r_{-}I^{-1}y)\\
\nonumber&=& (r_{-}I^{-1}x) \circ_A y+x \circ_A (r_{-}I^{-1}y)+x \circ_A y.
\end{eqnarray}
Therefore, we have
\begin{eqnarray*}
P(P(x) \circ_A y+x\circ_A P(y) +\lambda x\circ_A y)&=& \lambda^{2} r_{-}I^{-1}\Big((r_{-}I^{-1}x) \circ_A y+x \circ_A (r_{-}I^{-1}y)+x \circ_A y \Big)\\
&=& \lambda^{2}r_{-}(I^{-1}x \cdot_{r} I^{-1}y )\\
&=& \lambda^{2}(r_{-}I^{-1}x \circ_A r_{-}I^{-1}y )\\
&=& P(x)\circ_A P(y),
\end{eqnarray*}
which implies that $P$ is a Rota-Baxter operator of weight $\lambda$ on $A$.

Next we show that $(A,\circ_A,P,\omega_{I})$ is a quadratic Rota-Baxter Zinbiel algebra. Since $I^{\ast}=-I$, we have
$$ \omega_{I}(x,y)=\langle I^{-1}x,y \rangle=-\langle x,I^{-1}y \rangle=-\omega_{I}(y,x),$$
 which means that $ \omega_{I}$ is skew-symmetric.

Since the skew-symmetric part $a$ of $r$ is $(L,L+R)$-invariant, by Proposition \ref{invariance2},
we have $ I^{-1} \circ L_x =-(L_x^*+R_x^*)\circ I^{-1}$. Thus
\begin{eqnarray*}
\omega_{I}(x\circ_A y,z)-\omega_{I}(y,x\ast_A z)
&=&\langle I^{-1}(x\circ_A y),z  \rangle-\langle I^{-1}(y),x\ast_A z \rangle\\
&=&\langle I^{-1}\circ L_x(y)+(L_x^*+R_x^*) \circ I^{-1}(y),z \rangle\\
&=&0,
\end{eqnarray*}
which implies that \eqref{quadratic-condition1} holds.

Moreover, by using $r^{\ast}_{-}=r_{+}$ and $I=r_{+}-r_{-}$, we have
\begin{eqnarray*}
\omega_{I}(x,Py)+\omega_{I}(Px,y)+\lambda \omega_{I}(x,y)&=&\lambda\Big(\langle I^{-1}(x), r_{-}I^{-1}(y) \rangle+\langle  I^{-1} r_{-} I^{-1}(x),y \rangle+\langle I^{-1}x,y \rangle \Big)\\
&=& \lambda\langle(-I^{-1} r_{+} I^{-1}+I^{-1} r_{-} I^{-1}+ I^{-1})(x),y \rangle\\
&=& 0,
\end{eqnarray*}
which implies that \eqref{compatibility condition} holds.

Therefore, $(A,\circ_A,P,\omega_{I})$ is a quadratic Rota-Baxter Zinbiel algebra of weight $\lambda$.
\end{proof}

It is straightforward to check that if $P:A \to A$ is a Rota-Baxter operator of weight $\lambda$ on a Zinbiel algebra $(A,\circ_A)$, then
\begin{eqnarray}
\widetilde{P}:=-\lambda\Id-P
\end{eqnarray}
is also a Rota-Baxter operator of weight $\lambda$.
\begin{cor}
Let $(A,A_{r}^{\ast})$ be a factorizable Zinbiel bialgebra with $I=r_{+}-r_{-}$. Then $(A,\circ_A,\widetilde{P},\omega_{I})$ is also a quadratic Rota-Baxter Zinbiel algebra of weight $\lambda$, where $\widetilde{P}=-\lambda\Id-P=-\lambda r_{+}\circ I^{-1}$ and $\omega_{I}\in \wedge^{2} A^{\ast}$ is defined by
\eqref{SI}.
\end{cor}
\begin{proof}
Since $(A,A_{r}^{\ast})$ is a factorizable Zinbiel bialgebra,  by Theorem \ref{Factorizable Zinbiel algebra},  $P$ satisfies the compatibility  condition \eqref{compatibility condition}.  Thus we have
\begin{eqnarray*}
&&\omega_{I}(x,\widetilde{P}y)+\omega_{I}(\widetilde{P}x,y)+\lambda \omega_{I}(x,y)\\
&=& -\omega_{I}(x,Py)-\lambda\omega_{I}(x,y)-\omega_{I}(Px,y)-\lambda \omega_{I}(x,y)+\lambda \omega_{I}(x,y)\\
&=& -\omega_{I}(x,Py)-\omega_{I}(Px,y)-\lambda\omega_{I}(x,y)\\
&=& 0.
\end{eqnarray*}
This implies that $(A,\widetilde{P},\omega_{I})$ is  a quadratic Rota-Baxter Zinbiel algebra of weight $\lambda$.
\end{proof}

\begin{cor}\label{Zinbiel bialgebra isomorphism}
Let $(A,A_{r}^{\ast})$ be a factorizable Zinbiel bialgebra with $I=r_{+}-r_{-}$ and $P=\lambda r_{-}\circ I^{-1}$  the induced Rota-Baxter operator of weight $\lambda$. Then $((A_{P},\cdot_{P}),(A^{\ast},\cdot_{I}))$ is a Zinbiel bialgebra, where
$$ \xi \cdot_{I} \eta:= -\lambda I^{-1}\Big((\frac{1}{\lambda} I \xi)\circ_A(\frac{1}{\lambda} I \eta)\Big),\quad \forall \xi,\eta\in A^{\ast},~\lambda\neq 0.  $$
Moreover, $\frac{1}{\lambda} I:A^{\ast} \longrightarrow A$ gives a Zinbiel bialgebra isomorphism from $(A_{r}^{\ast},A)$ to $((A_{P},\cdot_{P}),(A^{\ast},\cdot_{I}))$.
\end{cor}

\begin{proof}
By Example \ref{dual Zinbiel alg}, we know that if $(A,A_{r}^{\ast})$ is a Zinbiel bialgebra, then $(A_{r}^{\ast},A)$ is also a Zinbiel bialgebra.
First we show that $\frac{1}{\lambda} I:A_{r}^{\ast} \longrightarrow A_{P}$ is a Zinbiel algebra isomorphism. In fact, for any $\xi,\eta \in A^{\ast},$ taking $x=I\xi\in A$ and $y=I\eta \in A$, by \eqref{RB operator in fac}, we have
\begin{eqnarray*}
\frac{1}{\lambda}I (\xi\cdot_{r}\eta)=\frac{1}{\lambda^{2}}( PI\xi\circ_A I\eta+I\xi\circ_A PI\eta   +\lambda I\xi\circ_A I\eta)= (\frac{1}{\lambda}I\xi)\cdot_{P} (\frac{1}{\lambda}I\eta).
\end{eqnarray*}
So $\frac{1}{\lambda} I$ is a Zinbiel algebra isomorphism.

Since $(\frac{1}{\lambda}I)^{\ast}=-\frac{1}{\lambda}I$, we have
\begin{eqnarray*}
(\frac{1}{\lambda}I)^{\ast} (\xi \cdot_{I} \eta)=\Big((-\frac{1}{\lambda} I \xi)\circ_A (-\frac{1}{\lambda} I \xi)\Big)=(\frac{1}{\lambda}I)^{\ast}(\xi)\circ_A (\frac{1}{\lambda}I)^{\ast}(\eta),
\end{eqnarray*}
which means that $(\frac{1}{\lambda}I)^{\ast}=-\frac{1}{\lambda}I:(A^{\ast},\cdot_{I})\longrightarrow (A,\circ_A)$ is also a Zinbiel algebra isomorphism. Since $(A_{r}^{\ast},A)$ is a Zinbiel bialgebra, by Proposition \ref{pro:induced isomorphism}, the pair $((A_{P},\cdot_{P}),(A^{\ast},\cdot_{I}))$ is also a Zinbiel bialgebra. Obviously, $\frac{1}{\lambda} I$ is a Zinbiel bialgebra isomorphism.
\end{proof}

By Theorems \ref{thm:Zin-Connes-RB} and   \ref{Factorizable Zinbiel algebra}, we have
\begin{cor}
    Let $(A,A_{r}^{\ast})$ be a factorizable Zinbiel bialgebra with $I=r_{+}-r_{-}$. Then $(A,\ast_{A},P,\omega_I)$ is a Rota-Baxter commutative associative algebra with a nondegenerate Connes cocycle $\omega_I$ of weight $\lambda$, where the linear map $P:A\longrightarrow A$ and $\omega_{I}\in \wedge^{2} A^{\ast}$ are defined by \eqref{P} and \eqref{SI}, respectively.
\end{cor}

\begin{ex}\label{Zinbiel double}{\rm
Let $(\frkd,\frkd_r^{\ast})$ be a factorizable Zinbiel bialgebra given in Theorem \ref{Z-double-fac}.
By Theorem \ref{Factorizable Zinbiel algebra},  $(\frkd,\cdot_{\frkd},P,\omega_I)$ is a quadratic Rota-Baxter Zinbiel algebra of weight $\lambda$,  where
\begin{eqnarray*}
P(x,\xi)&=&\lambda r_{-}\circ I^{-1} (x,\xi)=-\lambda(x,0),\\
\omega_I(x+\xi,y+\eta)&=& \langle I^{-1}(x+\xi),y+\eta \rangle=\langle -x+\xi,y+\eta \rangle=\xi(y)-\eta(x),
\end{eqnarray*}
for all $x,y\in A, \xi,\eta\in A^*$. Note that
$\omega_I$ is exactly the bilinear form given by {\rm(\ref{natural-bilinear-form})}.
Furthermore, $(\frkd,[\cdot,\cdot]_{\frkd},P,\omega_I)$ is a Rota-Baxter commutative associative algebra with a nondegenerate Connes cocycle $\omega_I$ of weight $\lambda$, where the commutative associative multiplication $\ast_{\frkd}$ is given by
$$ (x+\xi)\ast_{\frkd} (y+\eta)=x \ast_A y-\frkL^*_\xi y-\frkL^*_\eta x+
\xi \ast_{A^*} \eta-L^{\ast}_{x}\eta-L^{\ast}_{y} \xi, \quad \forall x,y\in A,\xi,\eta\in A^*. $$
}
\end{ex}

At the end of this section, we show that a quadratic Rota-Baxter Zinbiel algebra of nonzero weight can  give rise to a factorizable Zinbiel bialgebra.

\begin{thm}\label{thm:QRB-facpre}
    Let $(A,P,\omega)$ be a quadratic Rota-Baxter Zinbiel algebra of weight $\lambda$ $(\lambda\neq 0),$ and
    $ \huaI_{\omega}: A^{\ast}\longrightarrow A $ the induced linear isomorphism given by $\langle\huaI_{\omega}^{-1}x,y \rangle :=\omega(x,y).$ Then $r \in A\otimes A $ defined by
  \begin{equation}\label{eq:equiv} r_{+}:=\frac{1}{\lambda} (P+\lambda
\Id)\circ \huaI_{\omega}:A^{\ast} \longrightarrow A, \quad
r_{+}(\xi)=r(\xi,\cdot), \quad \forall \xi\in A^{\ast}
\end{equation}
    satisfies $\llbracket r,r\rrbracket=0$ and thus gives rise to a factorizable Zinbiel bialgebra $(A,A_{r}^{\ast})$.
\end{thm}
\begin{proof}
    Since $\omega$ is skew-symmetric, we have $\huaI_{\omega}=-\huaI_{\omega}^{\ast}$. By the fact that $\omega(x,Py)+\omega(Px,y)+\lambda \omega(x,y)=0$  for all $x,y\in A$, we have
    $$ \langle \huaI_{\omega}^{-1}x,Py\rangle+\langle \huaI_{\omega}^{-1}\circ P(x),y \rangle+\lambda\langle \huaI_{\omega}^{-1}x,y \rangle=0, $$
    which implies that $P^{\ast}\circ \huaI_{\omega}^{-1}+\huaI_{\omega}^{-1}\circ P+\lambda \huaI_{\omega}^{-1}=0, $ and then
    $$ \huaI_{\omega}\circ P^{\ast}+P\circ \huaI_{\omega}+\lambda \huaI_{\omega}=0.$$
Thus we have
    $$ r_{-}:=r_{+}^{\ast}=\frac{1}{\lambda}(-\huaI_{\omega}\circ P^{\ast}-\lambda \huaI_{\omega})=\frac{1}{\lambda} P\circ \huaI_{\omega}, $$
    and $\huaI_{\omega}=r_{+}-r_{-}$. Define a multiplication $\cdot_r$ on $A^{\ast}$ by
    $$\xi \cdot_{r} \eta= -(L_{r_+(\xi)}^*+R_{r_+(\xi)}^*)\eta+R_{r_{-}(\eta)}^* \xi.$$
Now we show that the following equation holds:
    \begin{eqnarray}\label{Zinbiel alg iso}
        \frac{1}{\lambda} \huaI_{\omega}(\xi \cdot_{r} \eta)=(\frac{1}{\lambda}\huaI_{\omega}\xi)\cdot_P(\frac{1}{\lambda}\huaI_{\omega}\eta).
    \end{eqnarray}
By the fact that $\omega(x\circ_A y,z)-\omega(y,x \ast_A z)=0$ for all $x,y,z\in A$, we have
$$ \langle \huaI_{\omega}^{-1}\circ L_x(y),z \rangle+\langle (L_{x}^*+R_{x}^*)\circ \huaI_{\omega}^{-1}(y),z \rangle=0, $$
which implies that  $ L_x\circ \huaI_{\omega}=-\huaI_{\omega} \circ (L_{x}^*+R_{x}^*).$ Therefore, by Proposition
\ref{invariance2}, the skew-symmetric part $a$ of $r$ is $(L,L+R)$-invariant.

On the one hand, for all $\xi,\eta\in A^*$, we have
\begin{eqnarray*}
    \huaI_{\omega}(\xi\cdot_r \eta)&=&\huaI_{\omega}\Big(-(L_{r_+(\xi)}^*+R_{r_+(\xi)}^*)\eta+R_{r_{-}(\eta)}^* \xi \Big)\\
    &=& L_{r_{+}(\xi)} \circ \huaI_{\omega}(\eta)+\huaI_{\omega}\circ R_{r_{-}(\eta)}^{\ast}(\xi)\\
    &=& L_{r_{+}(\xi)} \circ \huaI_{\omega}(\eta)+ R_{r_{-}(\eta)}\circ \huaI_{\omega}(\xi)\\
    &=&r_+(\xi)\circ_A (r_{+}(\eta)-r_{-}(\eta))+(r_{+}(\xi)-r_{-}(\xi))\circ_A r_-(\eta) \\
    &=& r_{+}(\xi)\circ_A r_{+}(\eta)-r_{-}(\xi)\circ_A r_{-}(\eta).
\end{eqnarray*}
On the other hand, we have
\begin{eqnarray*}
    &&(\huaI_{\omega}\xi)\cdot_{P}(\huaI_{\omega}\eta)\\
    &=&(P\huaI_{\omega}\xi)\circ_A(\huaI_{\omega}\eta)+(\huaI_{\omega}\xi)\circ_A(P\huaI_{\omega}\eta)
    +\lambda(\huaI_{\omega}\xi)\circ_A(\huaI_{\omega}\eta)\\
    &=& \lambda (r_{-}\xi)\circ_A(r_{+}\eta-r_{-}\eta)+\lambda(r_{+}\xi-r_{-}\xi)\circ_A (r_{-}\eta)+\lambda(r_{+}\xi-r_{-}\xi)\circ_A(r_{+}\eta-r_{-}\eta)\\
    &=& \lambda r_{+}(\xi)\circ_A r_{+}(\eta)-\lambda r_{-}(\xi)\circ_A r_{-}(\eta),
\end{eqnarray*}
which implies that  \eqref{Zinbiel alg iso} holds. Thus $\cdot_{r}$ is a Zinbiel multiplication and $\frac{1}{\lambda} \huaI_{\omega}$ is a Zinbiel algebra isomorphism from $(A^{\ast},\cdot_r)$ to $(A,\cdot_P)$.

 Finally, by the fact that $P+\lambda\Id,P:(A,\cdot_P)\longrightarrow (A,\circ_A)$ are both Zinbiel algebra homomorphisms, we deduce that
$$   r_{+}:=\frac{1}{\lambda} (P+\lambda \Id)\circ \huaI_{\omega},  \quad r_{-}:=\frac{1}{\lambda} P\circ \huaI_{\omega}:(A^{\ast},\cdot_r) \longrightarrow (A,\circ_A)$$
    are both Zinbiel algebra homomorphisms. Therefore, by Theorem \ref{heart}, we have $ \llbracket r,r\rrbracket =0$ and $(A,A_{r}^{\ast})$ is a quasi-triangular Zinbiel bialgebra. Since $\huaI_{\omega} = r_{+}-r_{-}$  is an isomorphism, the Zinbiel bialgebra $(A,A_{r}^{\ast})$ is factorizable.
\end{proof}


\begin{thebibliography}{a}
\bibitem{AOK}
J. K. Adashev, B. A. Omirov and A. Khudoyberdiyev, Classification of some classes of Zinbiel algebras. \emph{J. Gen. Lie Theory Appl.} 4 (2010).

\bibitem{ALO}
J. K. Adashev, M. Ladra and B. A. Omirov, The classification of naturally graded Zinbiel algebras with characteristic sequence equal to $(n-p,p)$. \emph{Ukrainian Math. J.} 71(2019), no.7, 985-1005.

\bibitem{Aguiar}
M. Aguiar, Pre-Poisson algebras. \emph{Lett. Math. Phys.} 54 (2000), 263-277.

\bibitem{Aguiar3}
M. Aguiar, On the associative analog of Lie bialgebras. \emph{J. Algebra} 244 (2001), 492-532.

\bibitem{AGMM}
D. V. Alekseevsky, J. Grabowski, G.  Marmo and P. W. Michor, Poisson structures on double Lie groups. \emph{J. Geom. Phys.} 26 (1998), 340-379.

\bibitem{Bai}
C. Bai, Left-symmetric bialgebras and an analogue of the classical Yang-Baxter equation. \emph{Commun. Contemp. Math.} 10 (2008), 221-260.

\bibitem{Bai2}
C. Bai, Double constructions of Frobenius algebras, Connes cocycles and their duality. \emph{J. Noncommut. Geom.} 4 (2010), no.4, 475-530.

%\bibitem{BGLM}C. Bai, L. Guo, G. Liu and T. Ma, Rota-Baxter Lie bialgebras, classical Yang-Baxter equations and special $L$-dendriform bialgebras. arXiv:2207.08703.

%\bibitem{BGM} C. Bai, L. Guo  and T. Ma, Bialgebras, Frobenius algebras and associative Yang-Baxter equations for Rota-Baxter algebras. arXiv:2112.10928.

\bibitem{BGN}
C. Bai, L. Guo and X. Ni, Generalizations of the classical Yang-Baxter equation and $\huaO$-operators.
\emph{J. Math. Phys.} 52 (2011), no.6, 063515.

%\bibitem{BGN2010}
%C. Bai, L. Guo and  X. Ni, Nonabelian generalized Lax pairs, the classical Yang-Baxter equation and post-Lie algebras. \emph{Comm. Math. Phys.} 297 (2010), 553-596.

\bibitem{Balavoine}
D. Balavoine, Homology and cohomology with coefficients, of an algebra over a quadratic operad.
\emph{J. Pure Appl. Algebra} 132 (1998), no.3, 221-258.

\bibitem{Ba}
G. Baxter, An analytic problem whose solution follows from a simple algebraic identity. \emph{Pacific J. Math.} {10}  (1960), 731-742.

%\bibitem{Bor}
%M. Bordemann, Generalized Lax pairs, the modified classical Yang-Baxter equation, and affine geometry of Lie groups. \emph{Comm. Math. Phys.} 135 (1990), 201-216.

\bibitem{Tower2}
M. Ceballos and D. A. Towers, Abelian subalgebras and ideals of maximal dimension in Zinbiel algebras.
\emph{Comm. Algebra} 51 (2023), no.4, 1323-1333.


\bibitem{Chapoton}
F. Chapoton, Zinbiel algebras and multiple zeta values. \emph{Doc. Math.} 27 (2022), 519-533.

\bibitem{CK}
A. Connes and D. Kreimer, Renormalization in quantum field theory and the Riemann-Hilbert problem. I. The Hopf algebra structure of graphs and the main theorem. \emph{Comm. Math. Phys.} 210 (2000), 249-273.

\bibitem{CFLM}
S. Covez, M. Farinati, V. Lebed and D. Manchon, Bialgebraic approach to rack cohomology.
\emph{Algebr. Geom. Topol.} 23 (2023), no.4, 1551-1582.

\bibitem{GGZ}
X. Gao, L. Guo and Y. Zhang, cocommutative matching Rota-Baxter operators, shuffle products with decorations and matching Zinbiel algebras.  \emph{J. Algebra } 586 (2021), 402-432.

\bibitem{GK}
V. Ginzburg and M. Kapranov, Koszul duality for operads. \emph{Duke Math. J.} 76 (1994), no.1, 203-272.

\bibitem{G1}
M. Goncharov, On Rota-Baxter operators of non-zero weight arisen from the solutions of the classical Yang-Baxter equation. \emph{ Sib. El. Math. Rep.} 14 (2017), 1533-1544.

\bibitem{G2}
M. Goncharov, Rota-Baxter operators and non-skew-symmetric solutions of the classical Yang-Baxter equation on quadratic Lie algebra. \emph{Sib. El. Math. Rep.} 16 (2019), 2098-2109.

\bibitem{G3}
M. Goncharov and V. Gubarev, Double Lie algebras of nonzero weight. \emph{Adv. Math.} 409 (2022), 108680.

\bibitem{Goncharov}
M. E. Goncharov and P. S. Kolesnikov, Simple finite-dimensional double algebras. \emph{J. Algebra} 500 (2018), 425-438.

\bibitem{Guo}
L. Guo,  An introduction to Rota-Baxter algebra. Surveys of Modern Mathematics, 4. International Press, Somerville, MA; Higher Education Press, Beijing, 2012.

\bibitem{Kosmann}
Y. Kosmann-Schwarzbach, Lie bialgebras, Poisson Lie groups and dressing transformation. In: Kosmann-
Schwarzbach, K. M. Tamizhmani, B. Grammaticos (eds) Integrability of Nonlinear Systems. Lecture Notes in Physics, vol 638. Springer, Berlin, Heidelberg, 2004, 107-173.

\bibitem{Ku}
B. A. Kupershmidt, What a classical $r$-matrix really is. \emph{J. Nonlinear Math. Phys.} 6 (1999), 448-488.

\bibitem{Lang}
H. Lang and Y. Sheng, Factorizable Lie bialgebras, quadratic Rota-Baxter Lie algebras and Rota-Baxter Lie bialgebras. \emph{Comm. Math. Phys.} 397 (2023), 763-791.

\bibitem{LB}
G. Liu and C. Bai, New splittings of operations of Poisson algebras and transposed Poisson algebras and related algebraic structures. \emph{STEAM-H: Sci. Technol. Eng. Agric. Math. Health.} 2023, 49-96.

\bibitem{Loday}
J. L. Loday, Cup-product for Leibniz cohomology and dual Leibniz algebras. \emph{Math. Scand.} 77(1995), no.2, 189-196.

\bibitem{LV}
J. L. Loday and B. Vallette, Algebraic operads. \emph{Grundlehren Math. Wiss.} 346, Springer, Heidelberg, 2012.

\bibitem{Mon}
S. Montgomery, Hopf algebras and their actions on rings, \emph{Amer. Math. Soc.}, Regional Conf. Ser. in Math., 82, 1993.

\bibitem{MS}
G. Mukherjee and R. Saha, Cup-product for equivariant Leibniz cohomology and Zinbiel algebras. \emph{Algebra Colloquium} 26 (2019), no.2, 271-284.

\bibitem{NB}
X. Ni and C. Bai, Poisson bialgebras. \emph{J. Math. Phys.} 54 (2013), no.2, 023515.

\bibitem{RS}
N. Reshetikhin and M. A. Semenov-Tian-Shansky, Quantum $R$-matrices and factorization problems. \emph{J. Geom. Phys.} 5 (1988), 533-550.

\bibitem{Saha}
R. Saha, Cup-product in Hom-Leibniz cohomology and Hom-Zinbiel algebras. \emph{Comm. Algebra} 48 (2020), no.10, 4224-4234.

\bibitem{STS}
M. A. Semenov-Tian-Shansky, What is a classical $r$-matrix? \emph{Funct. Anal. Appl.} 17 (1983), 259-272.

\bibitem{S2}
M. A. Semenov-Tian-Shansky, Integrable systems and factorization problems. \emph{Operator Theory: Advances and Applications} 141 (2003), 155-218.

\bibitem{SW}
Y. Sheng and Y. Wang, Quasi-triangular and factorizable antisymmetric infinitesimal bialgebras.
\emph{J. Algebra } 628 (2023), 415-433.

%\bibitem{WBLS}
%Y. Wang, C. Bai, J. Liu and Y. Sheng, Quasi-triangular pre-Lie bialgebras, factorizable pre-Lie bialgebras and Rota-Baxter pre-Lie algebras.  \emph{J. Geom. Phys.} 199 (2024), no. 105146.

\bibitem{Tower1}
D. A. Towers, Zinbiel algebras are nilpotent. \emph{J. Algebra Appl.} 22 (2023), no.8, 2350166.


\bibitem{Yau}
D. Yau, Deformation of dual Leibniz algebra morphisms. \emph{Comm. Algebra} 35 (2007), no.4, 1369-1378.

\bibitem{Zhe}
V. N. Zhelyabin, Jordan bialgebras and their connection with Lie bialgebras. (Russian) \emph{Algebra i Logika} 36 (1997), no. 1, 3-25, 117; translation in \emph{Algebra and Logic} 36 (1997), no. 1, 1-15.

\end{thebibliography}
\end{document}